\documentclass[11pt]{article}
\usepackage[final]{epsfig}
\usepackage[left=3cm,right=3cm, top=2.5cm,bottom=2.5cm,bindingoffset=0cm]{geometry}
\usepackage{amsmath}
\usepackage{amsfonts}
\usepackage{amsthm}
\usepackage{latexsym}
\usepackage{amssymb}
\usepackage{mathabx}
\usepackage{graphicx, tikz-cd}
\usepackage{epstopdf, setspace}
\usepackage{booktabs}
\usepackage{multicol,multirow, enumitem}
\usepackage{hyperref, mathtools}
\newlist{longenum}{enumerate}{5}
\setlist[longenum,1]{label=\roman*)}
\setlist[longenum,2]{label=\alph*)}

\usepackage{calrsfs}
\DeclareMathAlphabet{\pazocal}{OMS}{zplm}{m}{n}

\usepackage{tikz}
\usetikzlibrary{positioning}
\usetikzlibrary{decorations.text}
\usetikzlibrary{decorations.pathmorphing}
\usetikzlibrary{decorations.markings}

\usetikzlibrary{arrows} 
\tikzset{
    >=stealth',
    pil/.style={
           ->,
           thick,
           shorten <=2pt,
           shorten >=2pt,}
}
\tikzset{->-/.style={decoration={
  markings,
  mark=at position .7 with {\arrow{>}}},postaction={decorate}}}
  \tikzset{a/.style={decoration={
  markings,
  mark=at position .52 with {\arrow{angle 90}}},postaction={decorate}}}
\tikzset{-<-/.style={decoration={
  markings,
  mark=at position .4 with {\arrow{<}}},postaction={decorate}}}  

\newtheorem{innercustomthm}{Theorem}
\newenvironment{customthm}[1]
  {\renewcommand\theinnercustomthm{#1}\innercustomthm}
  {\endinnercustomthm}

\newtheorem{lemma}{Lemma}[section]
\newtheorem{proposition}[lemma]{Proposition}
\newtheorem{remark-definition}[lemma]{Remark-Definition}
\newtheorem{example}[lemma]{Example}
\newtheorem{theorem}[lemma]{Theorem}
\newtheorem{corollary}[lemma]{Corollary}

\newtheorem{proposition-conjecture}[lemma]{Proposition-conjecture}
\newtheorem{problem}[lemma]{Problem}

\theoremstyle{definition}

\newtheorem{definition}[lemma]{Definition}
\newtheorem{remark}[lemma]{Remark}

\begin{document}
\newcommand{\eps}{{\varepsilon}}
\newcommand{\proofend}{\hfill$\Box$\bigskip}
\newcommand{\C}{{\mathbb C}}
\newcommand{\Q}{{\mathbb Q}}
\newcommand{\R}{{\mathbb R}}
\newcommand{\Z}{{\mathbb Z}}
\newcommand{\RP}{{\mathbb {RP}}}
\newcommand{\CP}{{\mathbb {CP}}}
\newcommand{\PP}{{\mathbb {P}}}
\newcommand{\ep}{\epsilon}
\newcommand{\G}{{\Gamma}}

\newcommand{\SDiff}{\mathrm{SDiff}}
\newcommand{\SVect}{\mathfrak{s}\mathfrak{vect}}
\newcommand{\vect}{\mathfrak{vect}}

\newcommand{\Ker}[1]{\mathrm{Ker} \, #1}
\newcommand{\grad}[1]{\mathrm{grad} \, #1}
\newcommand{\sgrad}[1]{\mathrm{sgrad} \, #1}
\newcommand{\St}[1]{\mathrm{St} \, #1}
\newcommand{\rank}[1]{\mathrm{rank} \, #1}
\newcommand{\codim}[1]{\mathrm{codim} \, #1}
\newcommand{\corank}[1]{\mathrm{corank} \, #1}
\newcommand{\sgn}[1]{\mathrm{sgn} \, #1}
\newcommand{\ann}[1]{\mathrm{ann} \, #1}
\newcommand{\ind}[1]{\mathrm{ind} \, #1}
\newcommand{\Ann}[1]{\mathrm{Ann} \, #1}
\newcommand{\ls}[1]{\mathrm{span} \langle #1 \rangle}
\newcommand{\Tor}[1]{\mathrm{Tor} \, #1}
\newcommand{\diff}[1]{\mathrm{d}  #1}
\newcommand{\diffFX}[2]{ \dfrac{\partial #1}{\partial #2} }
\newcommand{\diffFXp}[2]{ \dfrac{\diff #1}{\diff #2} }
\newcommand{\diffX}[1]{ \frac{\partial }{\partial #1} }
\newcommand{\diffXp}[1]{ \frac{\diff }{\diff #1} }
\newcommand{\diffFXY}[3]{ \frac{\partial^2 #1}{\partial #2 \partial #3} }
\newcommand{\K}{\mathbb{K}}
\newcommand{\centrum}{\mathrm{Z}}
\newcommand{\Complex}{\mathbb{C}}
\newcommand{\Aut}{\mathrm{Aut}}
\newcommand{\Id}{\mathrm{E}}
\newcommand{\D}{\mathrm{D}}
\newcommand{\T}{\mathrm{T}}
\newcommand{\Cont}{\mathrm{C}}
\newcommand{\const}{\mathrm{const}}
\newcommand{\Hom}{\mathrm{H}}
\newcommand{\Ree}[1]{\mathrm{Re} \, #1}
\newcommand{\Imm}[1]{\mathrm{Im} \, #1}
\newcommand{\Tr}[1]{\mathrm{Tr} \, #1}
\newcommand{\tr}[1]{\mathrm{tr} \, #1}
\newcommand{\matrixtwobytwo}{\left(\begin{array}{|cc|}\hline 0 & 0 \\0 & 0 \\\hline \end{array}\right)}
\newcommand{\wave}{\tilde}
\newcommand{\LieBracket}{ [\, , ] }
\newcommand{\PoissonBracket}{ \{ \, , \} }
\newcommand{\g}{\mathfrak{g}}
\newcommand{\h}{\mathfrak{h}}
\newcommand{\lCal}{\mathfrak{l}}
\newcommand{\e}{\mathfrak{e}}
\newcommand{\so}{\mathfrak{so}}
\newcommand{\SO}{\mathrm{SO}}
\newcommand{\Orth}{\mathrm{O}}
\newcommand{\U}{\mathrm{U}}
\newcommand{\he}{\mathfrak{hso}}
\newcommand{\ELL}{\mathfrak{D}}
\newcommand{\hyp}{\mathfrak{D}^{h}}
\newcommand{\foc}{\mathfrak{D}^{\Complex}}
\newcommand{\sP}{\mathfrak{sp}}
\newcommand{\sL}{\mathfrak{sl}}
\newcommand{\ad}{\mathrm{ad}}
\newcommand{\Ad}{\mathrm{Ad}}
\newcommand{\zenter}{\mathrm{Z}}
\newcommand{\id}{\mathrm{id}}
\newcommand{\Ham}{\mathrm{Ham}}
\newcommand{\ham}{\mathfrak{ham}}
\newcommand{\Flux}{\mathrm{Flux}}
\newcommand{\Diffeo}{\mathrm{Diff}}
\newcommand{\halfTwist}{\mathrm{ht}}

\renewcommand{\proofname}{Proof}

\newcommand{\oneform}{\alpha}
\newcommand{\oneformtwo}{\beta}
\newcommand{\oneformthree}{\gamma}

\newcommand{\circulation}{ \pazocal C}
\newcommand{\circulationtwo}{ \circulation'}
\newcommand{\orbit}{\pazocal O}

\newcommand{\Fibr}{\pazocal{F}}
\newcommand{\Fibrtwo}{\pazocal{G}}

\newcommand{\Diff}{\pazocal{D}}

\newcommand{\MCG}{\mathrm{Mod}}
\newcommand{\Stab}{\mathrm{Stab}}

\newcounter{bk}
\newcommand{\bk}[1]
{\stepcounter{bk}$^{\bf BK\thebk}$%
\footnotetext{\hspace{-3.7mm}$^{\blacksquare\!\blacksquare}$
{\bf BK\thebk:~}#1}}

\newcounter{ai}
\newcommand{\ai}[1]
{\stepcounter{ai}$^{\bf AI\theai}$%
\footnotetext{\hspace{-3.7mm}$^{\blacksquare\!\blacksquare}$
{\bf AI\theai:~}#1}}



\title{Coadjoint orbits of symplectic diffeomorphisms of surfaces\\ and ideal hydrodynamics}

\author{Anton Izosimov\thanks{
Department of Mathematics,
University of Toronto, Toronto, ON M5S 2E4, Canada;
e-mail: {\tt izosimov@math.toronto.edu}, {\tt khesin@math.toronto.edu}
} , 
Boris Khesin*, 
and Mehdi Mousavi\thanks{
Department of Statistical and Actuarial Sciences, University of Western Ontario, London, ON N6A 5B7, Canada;
e-mail: \tt{smousav4@uwo.ca}
}
\\
}

\date{}
\maketitle
\begin{abstract} We give a classification of generic coadjoint orbits
for the groups of
symplectomorphisms and Hamiltonian diffeomorphisms of a closed
symplectic surface. We also classify simple Morse functions on
symplectic surfaces with respect to actions of those groups. This
gives an answer to V.Arnold's problem on describing all invariants of generic
isovorticed fields for the 2D ideal fluids.
For this we introduce a notion of anti-derivatives on a measured Reeb
graph and describe their properties.
\end{abstract}

\tableofcontents


\section{Introduction}
In this paper we classify generic coadjoint orbits of several diffeomorphism groups of surfaces. In particular, we answer a question about a complete set of invariants for generic
isovorticed fields in 2D ideal hydrodynamics posed by V.Arnold in \cite{AK}.

Recall that the corresponding classification problem for diffeomorphisms of the circle was solved 
by A.\,Kirillov in \cite{Kir1, kirillov1993orbit}. He showed that it is equivalent to classification of 
periodic quadratic differentials and described Casimirs for generic orbits. 
Orbits of the Virasoro--Bott group, a nontrivial extension of the circle diffeomorphism group, 
were classified independently in different terms by several authors, 
see \cite{Kir1, Seg1, Wit}. The latter problem is also equivalent to the classification 
of Hill's operators or projective structures on the circle. 
All those results deal with 
diffeomorphisms  of one-dimensional manifolds.

The classification problem for coadjoint orbits of area-preserving diffeomorphisms 
in two dimensions was known to specialists
for quite a while in view of its application in fluid dynamics, and it was explicitly formulated 
in \cite{AK}, Section I.5 in 1998. So far, to the best of our knowledge, the global classification was beyond reach and 
there were no results in this direction. 

In this paper we give an answer to this question by  describing the orbit classification 
 for symplectic and Hamiltonian diffeomorphisms of an arbitrary 2D oriented surface.

To obtain these classifications we first solve an auxiliary problem, which is of interest by itself: 
classify (and describe  invariants of) generic Morse functions on closed surfaces 
with respect to the action of area-preserving diffeomorphisms (possibly isotoped to the identity). 
It turns out that invariants of those actions on functions are given by the Reeb graphs of  
functions equipped with various collections of structures, such as a measure on the graph, 
homomorphisms of (local) homology groups of surfaces to that graph, a choice of a pants
 decomposition, and the flux across certain cycles as we describe in the corresponding sections.
Also the corresponding measures on Reeb graphs are not arbitrary but satisfy certain 
constraints in terms of asymptotic expansions at all three-valent vertices of the graph. 
To pass from classification of functions to classification of coadjoint orbits one needs 
to supplement the above data by the equality of appropriately defined circulation functions. 
\tikzstyle{vertex} = [coordinate]

 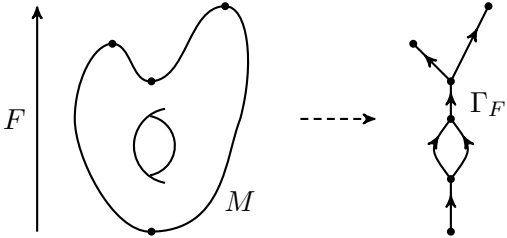
\begin{figure}[b]
\centerline{
\begin{tikzpicture}[thick]
\draw (0,1) .. controls (0,1.5) and (0.25,2) .. (0.5,2)
(0.5,2) ..controls (0.75,2) and (0.75,1.5) .. (1,1.5)
(2,2.5) ..controls (1.5,2.5) and (1.5,1.5) .. (1,1.5)
(2.2,1) .. controls (2.4,1.75) and (2.3,2.5) .. (2,2.5)
(1,-0.5) ..controls (0.5,-0.5) and (0,0.5) .. (0,1)
(1,-0.5) ..controls (2,-0.5) and (2,0.5) .. (2.2,1);
    \draw   (1.2,0.15) arc (260:100:0.5cm);
    \draw   (1,0.25) arc (-80:80:0.4cm);
\node [vertex] at (5,-0.5) (nodeA) {};
\node [vertex] at (5,0.2) (nodeB) {};
    \draw  [->-] (nodeA) -- (nodeB);
\node [vertex] at (5,1) (nodeC) {};
    \draw  [->-] (nodeB) .. controls +(-0.3,+0.4) .. (nodeC);
        \draw  [->-] (nodeB) .. controls +(0.3,+0.4) .. (nodeC);
            \node at (5.5,1.2) (nodeZ) {$\Gamma_F$};
\node [vertex] at (5,1.5) (nodeD) {};
    \draw  [->-] (nodeC) -- (nodeD);
\node [vertex] at (4.5,2) (nodeE) {};
    \draw  [->-] (nodeD) -- (nodeE);
\node [vertex] at (5.5,2.5) (nodeF) {};
    \draw  [->-] (nodeD) -- (nodeF);
\fill (nodeA) circle [radius=1.5pt];
\fill (nodeB) circle [radius=1.5pt];
\fill (nodeC) circle [radius=1.5pt];
\fill (nodeD) circle [radius=1.5pt];
\fill (nodeE) circle [radius=1.5pt];
\fill (nodeF) circle [radius=1.5pt];
\node [vertex] at (1.02,-0.5) (nodeAi) {};
\node [vertex] at (1.02,1.5) (nodeDi) {};
\node [vertex] at (0.5,2) (nodeEi) {};
\node [vertex] at (2,2.5) (nodeFi) {};
\fill (nodeAi) circle [radius=1.5pt];
\fill (nodeDi) circle [radius=1.5pt];
\fill (nodeEi) circle [radius=1.5pt];
\fill (nodeFi) circle [radius=1.5pt];
    \draw  [->] (-0.5,-0.5) -- (-0.5,2.5);
      \draw  [->, densely dashed] (3,1) -- (4,1);
    \node at (-0.8,1) (nodeA) {$F$};
        \node at (2.2,-0.1) (nodeA) {$M$};
\end{tikzpicture}
}
\caption{Reeb graph for a height function with two maxima on a torus.}\label{torusInt}
\end{figure}
\begin{example}\label{ex:intro}
{\rm
The following example outlines the basic constructions below.
The graph $\Gamma_F$, called the Reeb graph,  
is the set of connected components of the levels of a height function $F$
 on a  surface $M$, see Figure \ref{torusInt}. 
 Critical points of $F$ correspond to the vertices 
 of the graph $\Gamma_F$.
 This graph comes with a natural parametrization by the values of $F$. 
 For a symplectic surface $M$ its area form $\omega$ induces a measure $\mu$ on the graph, 
 which satisfies certain properties. For each edge $e\in \Gamma_F$ of the graph 
 $\Gamma_F$ one can consider the preimage $M_e\subset M$ bounded by the corresponding critical levels of $F$. Then  infinitely many moments   
 $$
 I_{\ell,e}(F):=\int_{M_e} \!\! F^\ell\,\omega, ~\ell=0,1,2,...
 $$
 of the function $F$  over each $M_e$  (or, equivalently, the moments of the induced function on each edge the graph) are invariants of the $\SDiff(M)$-action, i.e., the action on  the function $F$
  by symplectomorphisms of $M$. For the action of the group $\SDiff_0(M)$ consisting of symplectomorphisms in the connected component of the identity, 
  one encounters additional discrete invariants related to pants decompositions 
  and possible projections of the surface to the graph.
  For the group of Hamiltonian diffeomorphisms the above set of orbit invariants is supplemented 
  by fluxes of diffeomorphisms across certain cycles on the surface $M$.

In order to classify coadjoint orbits of the symplectomorphism group we
introduce  a notion of an anti-derivative, or circulation function,
for a Reeb graph. It turns out that such anti-derivatives
form a finite-dimensional space of dimension equal to the first Betti
number of the graph. Therefore the space of coadjoint orbits of the
symplectomorphism group of a surface is a bundle over the space of
fluid vorticities, where fiber coordinates can be thought of as
circulations, see details in Section \ref{sect:circ}.
}
\end{example}
 
\begin{table*}
\footnotesize
\begin{center}
\begin{tabular}{c|c|c|c}
classification    &  group action & invariants of the action &  see section \\
problem for & $G=$ & structure &  and theorem  \\
\hline
\hline
   &   &  &  \\
  \parbox{0.18\linewidth}{\centering Morse function $F$ on $M$} &     \parbox{0.15\linewidth}{\centering $\Diffeo(M)$, diffeomorphisms  of $M$}  &    \parbox{0.33\linewidth}{\centering $\Gamma_F$ -- Reeb graph} & \parbox{0.15\linewidth}{\centering Section \ref{MRG} Definition \ref{def:reeb}}\\
   &     &  &  \\
\hline
   &   &  &  \\
  \parbox{0.18\linewidth}{\centering Morse fibration $\pazocal F$ on $M$}  & \parbox{0.15\linewidth}{\centering $\SDiff(M),$ symplectomorphisms of~$M$}  & \parbox{0.33\linewidth}{\centering $\Gamma_{\pazocal F}$ -- weighted  graph: \\  graph + measure of edges \\ + expansions at vertices} & 
   \parbox{0.15\linewidth}{\centering Section \ref{sdiff-function-classif} Definition \ref{def:weighted} Theorem \ref{thmDTM} }\\
  &   &  &  \\
\hline
   &   &  &  \\
   \parbox{0.18\linewidth}{\centering Morse function $F$}  &    \parbox{0.15\linewidth}{\centering $\SDiff(M)$ }  &  \parbox{0.33\linewidth}{\centering $\Gamma_F$ -- Reeb graph  \\ + measure $\mu$ on the graph  (satisfying conditions at vertices)}  & 
  \parbox{0.15\linewidth}{\centering Section \ref{sdiff-function-classif} Definition \ref{MRGDef} Theorem \ref{thm:sdiff-functions}}\\
      &   &  &  \\
\hline
   &   &  &  \\
   \parbox{0.18\linewidth}{\centering Morse function $F$ genus$(M)=0$}     &     \parbox{0.15\linewidth}{\centering $\SDiff_0(M)$  $=\SDiff(M)$ $=\Ham(M)$ }   &    \parbox{0.33\linewidth}{\centering$\Gamma_F +  \mu$,\\  Reeb graph and measure} & 
   \parbox{0.15\linewidth}{\centering Section \ref{sdiff-function-classif} Corollary \ref{cor:sdiff-sphere}}    \\
 &   &  &   \\
  ---------------------  &   ---------------------   &   ---------------------  &  --------------------- \\
   \parbox{0.18\linewidth}{\centering Morse function $F$ genus$(M)=1$}     &      \parbox{0.15\linewidth}{\centering $\SDiff_0(M)$, connected component of $id\in \SDiff(M)$}  &
                                 \parbox{0.33\linewidth}{\centering $\Gamma_F +  \mu$ \\ + freezing homomorphism  $\pi_* \colon \Hom_1(\T^2,\Z)\to \Hom_1(\Gamma_F,\Z)$} &
                                 \parbox{0.15\linewidth}{\centering   Section \ref{sect:genus-one} Theorem \ref{Sdiff0Torus}} \\
 &   &  &  
 \\
  ---------------------  &   ---------------------   &   ---------------------  &  --------------------- \\
   \parbox{0.18\linewidth}{\centering Morse function $F$ genus$(M)\ge 2$ }    &  \parbox{0.15\linewidth}{\centering$\SDiff_0(M)$ } &
                     \parbox{0.33\linewidth}{\centering   $\Gamma_F +  \mu$  + higher freezing:  \\ 1) fixed pants decomposition \\ 2) $\Z^k_2$-valued half-twist invariant} &   \parbox{0.15\linewidth}{\centering Section \ref{sect:high-genus} Theorem \ref{Sdiff0action}}  \\
&       &   & 
  \\
\hline
   &   &  &  \\
  \parbox{0.18\linewidth}{\centering Morse function $F$  or  coadjoint \\ orbit of $F$,  genus$(M)=\varkappa$}  &   \parbox{0.15\linewidth}{\centering ${\rm Ham}(M)$}  &  \parbox{0.33\linewidth}{\centering $\Gamma_F +  \mu$ \\ + (higher) freezing \\ + flux on $\varkappa$ cycles}  &  \parbox{0.15\linewidth}{\centering Section \ref{HamSection} Theorem \ref{HamThm} } \\
   &      &     & 
   \\
\hline
   &   &  &  \\
  \parbox{0.18\linewidth}{\centering Morse coadjoint orbit of $[\alpha]\in \Omega^1/\diff\Omega^0$ }     &    \parbox{0.15\linewidth}{\centering $\SDiff(M)$ }    &     \parbox{0.33\linewidth}{\centering  $\Gamma_F +  \mu$ for function $F=\diff \alpha/\omega$\\ + circulation function $ \circulation=\int\alpha\in \Hom_1( \Gamma_F,\R)$ } &   
  \parbox{0.15\linewidth}{\centering Section \ref{sect:coadj-sdiff}  Theorem  \ref{thm:sdiffM}}  \\
&  &     &  \\
\hline
   &   &  &  \\
  \parbox{0.18\linewidth}{\centering Morse coadjoint orbit of $[\alpha]\in \Omega^1/\diff\Omega^0$ }     &    \parbox{0.15\linewidth}{\centering $\SDiff_0(M)$ }    &         \parbox{0.33\linewidth}{\centering $\Gamma_F +  \mu$ for  $F=\diff\alpha/\omega$ \\ + circulation function $\circulation$ \\ + (higher) freezing } & 
 \parbox{0.15\linewidth}{\centering Section \ref{sect:coadj-sdiff} Theorem  \ref{thm:sdiff0M} }    \\
&  &    &  \\
\hline
\\
\end{tabular}
\\
\end{center}
\caption{Sets of invariants in the classification problems studied below. }\label{table1}
\end{table*}

Table \ref{table1}  summarizes the main results of the paper presenting 
the set of invariants in each case studied below, as well as addressing 
the reader to the corresponding section and theorem for all the details. 
For a fixed closed 2D surface  $M$ equipped with an area form $\omega$ 
we classify the objects described in the first column of the table with respect to the group
listed in the second column. The invariants of the action are listed in the third column, and the 
corresponding reference to the classification theorem can be found in the last column.

In Section \ref{sect:euler} we describe motivation for this type of  classification problems coming 
from fluid dynamics. The case of manifolds with boundary and an application 
to computations of momenta of enstrophy will be discussed in a separate publication. 
Note that  the results of the paper can also be used to describe the manifolds of steady flows 
of the Euler equation, cf. the Choffrut--{\v{S}}ver{\'a}k description of a transversal slice
to special coadjoint orbits for  symplectomorphisms of an annulus \cite{ChSv}.

This also can be used for the extension of the orbit method to infinite-dimensional groups 
of 2D diffeomorphisms. According to this method, adjacency of  coadjoint orbits of a group
or its central extension mimics families of appropriate representations of the corresponding group. 
This methods turned out to be effective for affine groups and the Virasoro-Bott group, so one may hope to apply it to 2D diffeomorphisms and current groups as well. 

Finally, note that all objects in the present paper are infinitely smooth.  We remark on the case of finite smoothness at the end of the paper.

\begin{remark}
{\rm 
It is interesting to compare the description of $\SDiff(M)$-orbits for a surface $M$ with the classification of coadjoint orbits of the group $\Diffeo(S^1)$ of circle diffeomorphisms~\cite{Kir1, KhM}.
Its Lie algebra is ${\vect}(S^1)$  and the (smooth) dual space ${\vect}^*(S^1)$ is identified 
with the space of quadratic differentials
on the circle, $QD(S^1):=\{F(x)(dx)^2~|~F\in C^\infty(S^1, \R)\}$.
For a generic function $F$ changing sign on the circle, a complete set of invariants is given by the ``weights"
$$
I_{a_k}(F):=\int_{a_k}^{a_{k+1}}\!\!\sqrt{|F(x)|}\,dx
$$ 
of the quadratic differential between every two consecutive  zeros ${a_k}<{a_{k+1}}$ of $F(x)$ 
on the circle $S^1$. These orbits are of finite codimension equal to the number of zeros.
In a family of functions, where two new zeros, say $a'_k$ and $a''_k$, appear between original
zeros ${a_k}$ and ${a_{k+1}}$: ${a_k}<a'_k<a''_k <a_{k+1}$, one gains two extra Casimir functions, $I_{a'_k}$ and $I_{a''_k}$, and hence the codimension of the orbit jumps up by 2.

Similarly, for functions or coadjoint orbits of symplectomorphisms on
a 2D surface, the appearance
of a new pair of critical points, say, a saddle and a local maximum
for a function, leads
to splitting of one edge in two and, in addition to that, to
the appearance of a new edge in the corresponding Reeb graph, and
hence to two new families of Casimirs related to those extra edges, as
in Example \ref{ex:intro}.
}
\end{remark}


\bigskip

{\bf Acknowledgments.}
B.K. and A.I. were partially supported by  an NSERC research grant. B.K. is also grateful to the \'Ecole Polytechnique in Paris and the Erwin
Schr\"odinger Institute in Vienna for their hospitality and support. M.M. thanks Tudor Ratiu and the Centre Interfacultaire Bernoulli at EPFL  for their hospitality.


\bigskip

\section{The main setting and hydrodynamical motivation} \label{sect:reeb}

\subsection{Two classification problems}

Consider a closed 2D surface $M$ with a symplectic form $\omega$ on it. 
We are interested in  classifying generic coadjoint orbits of the groups  of symplectic and Hamiltonian diffeomorphisms of $M$. 

\smallskip

Namely, consider the following three groups: $\SDiff(M), \SDiff_0(M),$ and  ${\rm Ham}(M)$.
The group $\SDiff(M)$  consists of area-preserving (i.e. {\it symplectic}) diffeomorphisms of $M$.
Its connected component of the identity diffeomorphism is denoted by $\SDiff_0(M)$.
The group of {\it Hamiltonian} diffeomorphisms  ${\rm Ham}(M)$ consists of those symplectic 
diffeomorphisms of $M$ which can be connected with the identity by a Hamiltonian 
(non-autonomous) flow, i.e., by a flow of a vector field having a time-dependent Hamiltonian function. 

The first question we are going to address, and which we specify later, is

\begin{problem} \label{prob:orbits}
Classify generic coadjoint orbits of the  three groups $i)\, \SDiff(M),$ $ii) \,\SDiff_0(M),$ and  $iii)\, {\rm Ham}(M)$.
\end{problem}

It is closely related to the following problem. Let $F$ be a smooth Morse function on the surface $M$.

\begin{definition}
{\rm
A Morse function $F \colon M \to \R$ is called \textit{simple} if for each $a \in \R$, the corresponding level $F^{-1}(a)$ contains at most one critical point.
}
\end{definition}

\begin{problem}\label{prob:functions}
 Classify simple Morse functions on a symplectic surface $M$
with respect to the action of each of the three groups $i)\, \SDiff(M),$ $ii) \,\SDiff_0(M),$ and  $iii)\, {\rm Ham}(M)$.
\end{problem}


\subsection{Dual Lie algebras and coadjoint action} \label{sect:coadjoint}
In the general setting for an $n$-dimensional  manifold $M$ equipped 
with a volume form $\mu$ consider the  Lie group 
$G=\SDiff(M)$ of volume-preserving diffeomorphisms of $M$. 
The corresponding Lie algebra $\mathfrak g={\SVect}(M)$ 
consists of smooth divergence-free vector fields in $M$.  
(The same Lie algebra corresponds to the connected subgroup $\SDiff_0(M)\subset \SDiff(M)$.)
The natural smooth dual space for this Lie algebra is the space of  cosets $[\alpha]$
of smooth 1-forms $\alpha$ on $M$ modulo exact 1-forms, 
$\mathfrak g^*=\Omega^1(M)/\diff \Omega^0(M)$, see e.g. \cite{AK}. 
The groups $\SDiff(M)$ and $\SDiff_0(M)$ act on these 1-forms and their cosets by 
volume-preserving diffeomorphisms, i.e. by the corresponding change of coordinates.
This means that Problems 1(i) and 1(ii) reduce to description of invariants of 
cosets $[\alpha]$ of 1-forms 
on a surface $M$ with respect to area-preserving diffeomorphisms and those 
diffeomorphisms isotoped to the identity, respectively.

For  the Lie group of Hamiltonian diffeomorphisms ${\rm Ham}(M)$ its Lie algebra ${\ham}(M)$ 
consists of Hamiltonian vector fields on $M$. It can be identified with the Poisson algebra 
of  functions on $M$ modulo additive constants, i.e.  the Poisson algebra of Hamiltonians 
normalized by the condition of
zero mean: $\ham(M)=\{H\in C^\infty(M)~|~\int_MH\,\omega=0\}$.
For a surface $M$ the smooth dual space $\ham^*(M)$ can  be viewed as the space of 
exact 2-forms $\xi$ on $M$. Alternatively, it can also be identified with the space of  functions with zero mean: 
$\ham^*(M)=\{\xi=\rho\,\omega~|~\int_M\rho\,\omega=0\}$. Thus Problem~1(iii)
is equivalent to the classification problem of generic smooth functions with zero mean on a surface $M$ with respect to the action of Hamiltonian diffeomorphisms. 

To describe coadjoint invariants with respect to all these group actions we first
consider Problem \ref{prob:functions} of classification of generic smooth functions with respect to these actions.
\medskip


\subsection{Motivation: Euler equations, vorticity, Casimirs} \label{sect:euler}

The problem of classification of coadjoint orbits for volume-preserving diffeomorphisms is of particular importance in hydrodynamics. 
Let $M$ be an $n$-dimensional Riemannian manifold with a volume form $\mu$ 
and without boundary.
The  motion  of an inviscid incompressible fluid  on  $M$  is governed by  the 
classical Euler equation
\begin{equation}\label{idealEuler}
\partial_t v+(v, \nabla) v=-\nabla p
\end{equation}
describing an evolution of a  divergence-free velocity field $v$ of a fluid flow in $M$, 
${\rm div}\, v=0$. 
The pressure function $p$ entering the Euler equation is defined uniquely modulo 
an additive constant by this equation along with the divergence-free constraint 
on the velocity $v$. The term
$(v, \nabla) v$ stands for the Riemannian covariant derivative 
$\nabla_v v$ of the field $v$ along itself.

According to  Arnold's approach to the Euler equation \cite{Arn66}, the latter 
can be regarded as an 
equation of the geodesic flow on the group $\SDiff(M):=\{\phi\in \Diffeo~|~\phi^*\mu=\mu\}$ 
of volume-preserving diffeomorphisms of $M$ with respect to the right-invariant metric on the group 
given by the $L^2$-energy of the velocity field: $E(v)=\frac 12\int_M(v,v)\,\mu$. 
Then the Euler equation describes an evolution of the fluid velocity field $v(t)$, 
i.e. an evolution of a vector in the Lie algebra
${\SVect}(M)=\{v\in {\vect}(M) \mid L_v\mu=0\}$, 
tracing the geodesic on the group $\SDiff(M)$ defined by the initial condition $v(0)=v_0$.

\par

The geodesic point of view implies that the Euler equation has  the following Hamiltonian 
reformulation.  Consider the (regular) dual space $\mathfrak g^*={\SVect}^*(M)$ 
to the space $\mathfrak g={\SVect}(M)$ of divergence-free vector fields on $M$.
As mentioned above, this dual space $\mathfrak g^*$ has a natural  description 
as the  space of cosets  $\mathfrak g^*=\Omega^1(M) / \diff \Omega^0(M)$, 
where the coadjoint action of the group $\SDiff(M)$ on the dual 
 $\mathfrak g^*$ is given by the change of coordinates in (cosets of) 1-forms on $M$ 
 by means of volume-preserving diffeomorphisms.
 
Recall that the manifold $M$ is equipped with a Riemannian metric $(.,.)$, and it allows one to
identify the (smooth parts of) the Lie algebra and its dual by means of the so-called inertia operator:
given a vector field $v$ on $M$  one defines the 1-form $\alpha=v^\flat$ 
as the pointwise inner product with vectors of the velocity field $v$:
$v^\flat(W): = (v,W)$ for all $W\in T_xM$, see details in \cite{AK}. 
(Note that divergence-free fields $v$ correspond to co-closed 1-forms $v^\flat$.)
The Euler equation \eqref{idealEuler} rewritten on 1-forms is
$$
\partial_t \alpha+L_v \alpha=-dP\,
$$
for the 1-form $\alpha=v^\flat$ and an appropriate function $P$ on $M$.
In terms of the cosets of 1-forms $[\alpha]=\{\alpha+df\,|\,f\in C^\infty(M)\}\in \Omega^1(M) / \diff \Omega^0(M)$, the Euler equation looks as follows: 
\begin{equation}\label{1-forms}
\partial_t [\alpha]+L_v [\alpha]=0
\end{equation}
on the dual space $\mathfrak g^*$, where $L_v$ is the Lie derivative along the field $v$. 
The latter form has several important features.

First of all, the Euler equation \eqref{1-forms} on $\mathfrak g^*$ is a Hamiltonian equation. 
As the dual space to a Lie algebra, $\SVect^*(M)$ has the natural Lie-Poisson structure. The corresponding Hamiltonian operator is given by the Lie algebra coadjoint action ${\rm ad}^*_v$, which coincides with the Lie derivative in the case of the diffeomorphism group: ${\rm ad}^*_v=L_v$.
Its symplectic leaves are coadjoint orbits of the corresponding group $\SDiff(M)$.
The Euler equation is the Hamiltonian equation on the dual space $\mathfrak g^*$ with respect to this Lie-Poisson structure and with the Hamiltonian functional $H$ 
given by the fluid's kinetic energy,
$H([\alpha])=E(v)= \frac 12\int_M(v,v)\,\mu$ for $\alpha=v^\flat$, see details in \cite{Arn66, AK}.

Secondly, the equation form  \eqref{1-forms} shows that according to the Euler equation the 
coset of 1-forms $[\alpha]$ evolves by a volume-preserving change of coordinates, 
i.e. during the Euler evolution it remains in the same coadjoint orbit in $\mathfrak g^*$.
Introducing the {\it vorticity 2-form} $\xi:=dv^\flat$ as the differential of the 1-form 
$\alpha=v^\flat$ we note  that  the vorticity exact 2-form is well-defined for cosets $[\alpha]$: 
1-forms $\alpha$ in the same coset have equal vorticities $\xi=d\alpha$. 
The corresponding Euler equation assumes the vorticity (or Helmholtz) form
\begin{equation}\label{idealvorticity}
\partial_t \xi+L_v \xi=0\,,
\end{equation}
which means that the vorticity form is transported by (or ``frozen into") the fluid flow (Kelvin's theorem).

\begin{remark}
{\rm
The definition of vorticity $\xi$ as an exact 2-form $\xi=dv^\flat$ makes sense for a manifold $M$ 
of any dimension. In 3D the vorticity 2-form  is identified with the 
vorticity vector field $\hat\xi=\mathrm{curl}~v$  by means of the relation $i_{\hat\xi} \mu=\xi$ for the volume form $\mu$. 
In 2D  one  identifies the vorticity 2-form $\xi$ with a function $\hat\xi$ satisfying $\xi=\hat\xi\cdot\mu$.

In this paper we will be dealing with  two-dimensional oriented 
surfaces $M$ without boundary, while 
 the group $\SDiff(M)$ of volume-preserving diffeomorphisms of $M$ coincides with the group 
${\rm Symp}(M)$ of symplectomorphisms of $M$ with the area form $\mu=\omega$ given by the symplectic structure.
}
\end{remark}

\begin{remark}
{\rm
The fact that the vorticity  $\xi$ is ``frozen into" the incompressible flow allows one to define Casimirs, i.e., first integrals of the  hydrodynamical Euler equation valid for any Riemannian metric on $M$. These Casimirs are invariants of the coadjoint action of the corresponding group
$\SDiff(M)$.

In 2D the Euler equation on $M$ has infinitely many {\it enstrophy invariants}
$$
I_\lambda(\xi):=\int_M \!\lambda(\xi)\,\omega\,,
$$
where 
$\lambda(\xi)$ is an arbitrary function of vorticity. In particular,  the enstrophy momenta
$I_n(\xi):=\int_M {\xi}^n\,\omega$ are invariants for any $n\in \mathbb N$.
These invariants are  fundamental in the study of hydrodynamical stability of 2D flows, 
and in particular, were  the basis for Arnold's stability criterion, see \cite{Arn66, AK, Shn}. 
In the energy-Casimir method one studies the second variation of the energy functional 
with an appropriately chosen combination of Casimirs. 
A description of orbits by means of Casimirs also allows one to
construct asymptotic solutions to the Navier--Stokes equation,
localized near a curve or a surface \cite{MSh}, as well as to obtain a
precise structure of the set of Euler steady solutions \cite{ChSv}.

However, the functionals $I_\lambda$ do not form a complete set of Casimirs 
in either of the groups $\SDiff(M)$, $\SDiff_0(M)$, or ${\rm Ham}(M)$, even for the case $M = S^2$.
In the present paper we give a complete description of these invariants, namely, 
a complete classification of generic  coadjoint orbits of those groups, and hence, 
in particular, of generic vorticity functions $\xi$. 
Roughly speaking, for the symplectomorphism groups one
needs to consider analogs of the functionals $I_\lambda$ associated with every edge of a special
graph related to the vorticity function, as well as  a collection of discrete invariants.
For the group of Hamiltonian diffeomorphisms all these are supplemented by the flux type 
functionals over certain cycles on $M$, as will be discussed below. 
}
\end{remark}


\bigskip

\section{Simple Morse functions on symplectic surfaces}\label{functionsSection}

\subsection{Measured Reeb graphs}\label{MRG}

Throughout the paper let $M$ be a connected oriented two-dimensional surface without 
boundary, and let $F \colon M \to \R$ be a Morse function on $M$. 
Consider the space $\Gamma_F$ of connected components of $F$-levels with 
the induced quotient topology. This space is a finite connected graph, whose vertices 
correspond to critical levels of $F$.

\begin{definition}\label{def:reeb}
{\rm
This graph $\Gamma_F$ is called the \textit{Reeb graph}\footnote{This graph is also called the Kronrod graph of a function, see~\cite{Arnold58}.} of the function $F$.
}
\end{definition}

The function $F$ on $M$ descends to a function $f$ on the Reeb graph $ \Gamma_F$. In what follows, by a Reeb graph we always mean a pair: a graph, and a function on it.  It is also convenient to assume that $\Gamma_F$ is oriented: edges are oriented in the direction of increasing $f$. 

 \begin{figure}[t]
\centerline{
\begin{tikzpicture}[thick]
    \draw (2,2) ellipse (1cm and 2cm);
    \draw   (2.05,1.05) arc (225:135:1.2cm);
    \draw   (2.05,1.05) arc (-45:45:1.2cm);
    \draw   (1.1,1.05) arc (-120:-60:0.95cm);
    \draw   (2.05,1.05) arc (-120:-60:0.85cm);
    \draw  [densely dashed] (2.05,1.05) arc (60:120:0.95cm);
    \draw  [densely dashed] (2.9,1.05) arc (60:120:0.85cm);
    \draw   (1.1,2.75) arc (-120:-60:0.95cm);
    \draw  (2.05,2.75) arc (-120:-60:0.85cm);
    \draw  [densely dashed] (2.05,2.75) arc (60:120:0.95cm);
    \draw  [densely dashed] (2.9,2.75) arc (60:120:0.85cm);
    \draw  [->] (0.7,1) -- (0.7,3);
    \node at (0.5,2) (nodeA) {$F$};
        \node at (0.9,3.9) (nodeG) {$M$};
    \node at (7.7,2) (nodeB) {$\Gamma_F$};
    \node [vertex] at (7,0) (nodeC) {};
    \node [vertex]  at (7,1.05) (nodeD) {};
    \node [vertex] at (7,2.75) (nodeE) {};
    \node [vertex]  at (7,4) (nodeF) {};
    \node  at (2.1,4) (nodeG) {};
    \node  at (7,4) (nodeFdouble) {};
    \node  at (7,0) (nodeCdouble) {};
    \node  at (7,1.05) (nodeDdouble) {};
    \node at (7,2.75) (nodeEdouble) {};
    \node  at (2.1,0) (nodeH) {};
    \node  at (2.9,1.05) (nodeI) {};
    \node  at (2.9,2.75) (nodeJ) {};
    \draw  [->-] (nodeC) -- (nodeD);
    \fill (nodeC) circle [radius=1.5pt];
    \fill (nodeD) circle [radius=1.5pt];
    \fill (nodeE) circle [radius=1.5pt];
    \fill (nodeF) circle [radius=1.5pt];
    \draw  [->-] (nodeD) arc (-45:44:1.2cm);
    \draw  [->-] (nodeD) arc (225:136:1.2cm);
    \draw  [->-] (nodeE) -- (nodeF);
    \draw  [dashed, ->] (nodeG) -- (nodeFdouble);
    \draw  [dashed, ->] (nodeH) -- (nodeCdouble);
    \draw  [dashed, ->] (nodeI) -- (nodeDdouble);
    \draw  [dashed, ->] (nodeJ) -- (nodeEdouble);
\end{tikzpicture}
}
\caption{Reeb graph for a height function on a torus.}\label{torus}
\end{figure}
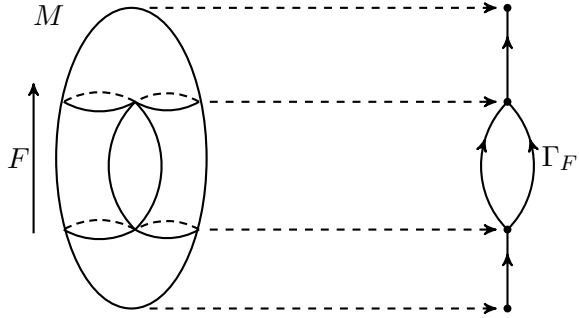

\begin{example}
{\rm
Figure \ref{torus} shows a torus and the Reeb graph of the height function on this torus.
}
\end{example}

We will confine ourselves to the study of {\it simple Morse functions} $F \colon M \to \R$,
i.e., those Morse functions whose critical values are all distinct. Such functions 
form an open dense set in $C^2$-topology among all smooth functions on $M$.

\begin{proposition}
Let $M$ be a closed connected 2D surface, and let  $F \colon M \to \R$  be a simple Morse function. Let also $\pi \colon M \to \Gamma_F$ be the natural projection.
Then:
\begin{longenum}
\item All vertices of $\Gamma_F$ are either $1$-valent or $3$-valent. 
\item  If $v$ is a $1$-valent vertex of $\Gamma_F$, then $\pi^{-1}(v)$ is a single point; this is a point of local minimum or local maximum of the function $F$. 
\item If $v$ is a $3$-valent vertex of $\Gamma_F$, then $\pi^{-1}(v)$ is a figure eight; the self-intersection point of this figure eight is a saddle critical point of $F$.
\item If $x$ is an interior point of some edge $e \subset \Gamma_F$, then $\pi^{-1}(x)$ is a circle.
\item For each $3$-valent vertex of $\Gamma_F$, there are either two incoming and one outgoing edge, or two outgoing and one incoming edge (see Figure \ref{trb}).
\item The first Betti number of the graph $\Gamma_F$ is equal to the genus of $M$.
\end{longenum}

\end{proposition}
 \begin{proof}
The proof follows from standard Morse theory considerations.
 \end{proof}
 
  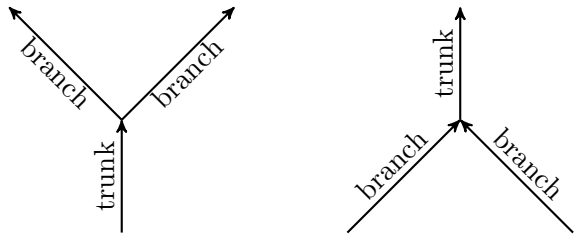
\begin{figure}[t]
\centerline{
\begin{tikzpicture}[thick]
\draw  [->] (2,0) -- (2,1.5);
\draw  [->] (2,1.5) -- (3.5,3);
\draw  [->] (2,1.5) -- (0.5,3);
\draw [decoration={text along path,
    text={trunk},text align={center},  raise = 0.1 cm},decorate]  (2,0) -- (2,1.5);
    \draw [decoration={text along path,
    text={branch},text align={center},  raise = -0.3 cm},decorate] (2,1.5) -- (3.5,3);
        \draw [decoration={text along path,
    text={branch},text align={center}, raise = -0.3cm},decorate]   (0.5,3) -- (2,1.5);
    \draw  [->] (5,0) -- (6.5,1.5);
       \draw  [->] (8,0) -- (6.5,1.5); 
              \draw  [->] (6.5,1.5) -- (6.5,3); 
              \draw [decoration={text along path,
    text={trunk},text align={center},  raise = 0.1 cm},decorate] (6.5,1.5) -- (6.5,3);
    \draw [decoration={text along path,
    text={branch},text align={center},  raise = 0.1 cm},decorate]  (5,0) -- (6.5,1.5);
        \draw [decoration={text along path,
    text={branch},text align={center}, raise = 0.1cm},decorate]    (6.5,1.5) -- (8,0); 
\end{tikzpicture}
}
\caption{Trunk and branches.}\label{trb}
\end{figure}
 
Assume that $e_0, e_1$, and $e_2$ are three edges of  $\Gamma_F$ which meet at a $3$-valent vertex $v$. Then $e_0$ is called the \textit{trunk} of $v$, and $ e_1, e_2$ are called \textit{branches} of $v$ if either $e_0$ is an outgoing edge for $v$, and $e_1, e_2$ are its incoming edges, or  $e_0$ is an incoming edge for $v$, and $e_1, e_2$ are its outgoing edges, see Figure \ref{trb}.

\medskip

Now, fix an area form $\omega$ on the surface $M$. Then the natural {\it projection map} $\pi:M\to \Gamma_F$ induces a measure $\mu$ on $\Gamma_F$. By definition, a set $X \subset \Gamma_F$ is measurable if and only if its preimage $\pi^{-1}(X) \subset M$ is measurable. For a measurable set $X \subset \Gamma_F$, its measure $\mu(X)$ equals the area of $\pi^{-1}(X)$.

\begin{proposition}\label{measureproperty} The measure $\mu$ has the following properties.
\begin{longenum}
\item Let $ [v,w]$ be an edge of $\Gamma_F$, and let $f(v) < f(w)$. 
Then there exists a function $\eta(z)$ smooth in the interval $(f(v),f(w))$ such that $\mu([v,x] ) = \eta(f(x))$ for any point $x \in (v,w)$, and $\eta'(z) \neq 0$ for any $z \in (f(v),f(w))$.
\item If, in addition, $v$ is a $1$-valent vertex, then $\eta(z)$ is smooth up to $f(v)$, and moreover, $\eta'(z) \neq 0$ at $z = f(v)$. Analogously, if $w$ is a $1$-valent vertex, then $\eta(z)$ is smooth up to  $f(w)$, and $\eta'(z) \neq 0$ at $z = f(w)$. 
\item Assume that $v$ is a $3$-valent vertex of $\Gamma_F$. Without loss of generality assume that $f(v)=0$ (if not, we replace $f$ by $\tilde f(x):=f(x)-f(v)$).
Let $e_0$ be the trunk of $v$, and let $e_1, e_2$ be the branches of $v$.
 Then there exist functions $ \psi, \eta_0, \eta_1,\eta_2$ of one variable, smooth in the neighborhood of the origin $0\in \R$ and such that for any point $x \in e_i$ sufficiently close to $v$, we have
\begin{align}\label{saddleAsymptotic}
\mu([v,x])&= \,\eps_i\psi(  f(x) )\ln |   f (x)| + \eta_i(   f(x) ),
\end{align}
where $\eps_0 = 2$, $\eps_1 = \eps_2 =  -1$,  $\psi(0) = 0$, $\psi'(0) \neq 0$,
and $\eta_0 + \eta_1 + \eta_2 =~0$.
\end{longenum}
 \end{proposition}
 
For the proof, we need two preliminary lemmas. The first of them is known as the Morse-Darboux lemma. This lemma is a particular case of {Le lemme de Morse isochore} due to Colin de Verdi{\`e}re and Vey \cite{CDV}.

\begin{lemma}\label{MDL}
Assume that $M$ is a two-dimensional manifold with an area form $\omega$, and let $F \colon M \to \R$ be a smooth function. Let also $O$ be a Morse critical point of $F$. Then there exists a chart $(p,q)$ centered at $O$ in $M$ such that  $\omega = \diff p \wedge \diff q$, and $F = \lambda \circ S $ where $ S  = p^2 + q^2$ or $ S = pq$. The function $\lambda$ of one variable is smooth in the neighborhood of the origin $0\in\R$, and $\lambda'(0) \neq 0$.
\end{lemma}

Note that since $\lambda'(0) \neq 0$, the function $S$ can be expressed in terms of $F$.
The latter allows one to descend $S$ to a locally defined function $s$ on the Reeb graph of $F$ near the corresponding vertex.

\par

The second lemma is due to Dufour, Molino, and Toulet \cite{DTM}.
\begin{lemma}\label{DMT1}
Let $F \colon M \to \R$ be a simple Morse function, and let $v$ be a $3$-valent vertex of $\Gamma_F$.
Let also $e_0$ be the trunk of $v$, and let $e_1, e_2$ be the branches of $v$. Then there exist functions $\zeta_0(z), \zeta_1(z), \zeta_2(z)$ of one real variable, smooth in the neighborhood of the origin, and such that for any point $x \in e_i$ sufficiently close to $v$, we have
\begin{align}\label{DMTFormulas}
\begin{aligned}
\mu([v,x])&= \eps_i s(x) \ln |  s(x) | + \zeta_i(  s(x) )
\end{aligned}
\end{align}
where $\eps_0 = -2\eps_1 = -2\eps_2 = \pm 2$,  $\zeta_0 + \zeta_1 + \zeta_2 = 0$, and 
the function $s(x)$,  well-defined on $\Gamma_F$ for $x$ sufficiently close to the vertex $v$, is obtained by descending the function $S$ from the preceding lemma.
\end{lemma}

\begin{remark}
Note that the latter lemma is formulated in \cite{DTM} in terms of infinite jets of functions $\zeta_i$, as the corresponding expansion involves
the functions $\zeta_0$ and $\zeta_{1,2}$ on different sides of the origin. The relation $\zeta_0 + \zeta_1 + \zeta_2 = 0$ for jets provides the existence
of the corresponding smooth functions defined in the whole neighborhood of the origin. The same holds for the functions $\psi$ and $\eta_i$ in Proposition  \ref{measureproperty}.
\end{remark}

\begin{proof}[Proof of Proposition \ref{measureproperty}]
Let us prove statement (i). The preimage of the open edge $(v,w)$ under the projection $\pi$ is a cylinder $\mathrm{Cyl}$. For any $z \in (f(v), f(w))$, the set $\{ P \in \mathrm{Cyl} : F(P) = z\}$ is a periodic trajectory of the Hamiltonian vector field $X_F = \omega^{-1}\diff F$. 
Denote the period of this trajectory by $T(z)$. Then a standard argument shows that
\begin{align}\label{period-formula}
\begin{aligned}
\mu([v,x]) = \int_{f(v)}^{f(x)} \!T(z)\diff z.
\end{aligned}
\end{align}
Now  statement (i) follows from the fact that $T(z)$ is a smooth non-vanishing function of $z$.

 To prove statement (ii)  it suffices to consider the case of a $1$-valent vertex $v$; 
 the second case is analogous. For a $1$-valent vertex $v$ the preimage of the interval 
 $[v, w)$ under the projection $\pi$ is an open disk $D \subset M$. The only critical point of $F$ in the disk $D$ is $ O = \pi^{-1}(v)$; the point $O$ is a non-degenerate minimum point. By Lemma \ref{MDL}, there exists a Darboux chart $(p,q)$ centered at the point $O$ such that $F =  \lambda \circ S$ where $S = p^2 + q^2$. Let 
 $$
 D_x =\pi^{-1}([v,x)) \subset D. 
 $$ 
 In $(p,q)$ coordinates, the set $D_x$ is a closed disk radius $\sqrt{s(x)} $ center $(0,0)$. Therefore,
$$
\mu([v,x]) = \int_{D_x} \!\!\omega = \pi s(x) = \pi \lambda^{-1}(f(x)),
$$
which implies the statement.
\par
The third statement is proved by substituting $s = \lambda^{-1} \circ f$ into formula \eqref{DMTFormulas}.
\end{proof}

\begin{remark}
{\rm
The statement of Proposition \ref{measureproperty} 
can be understood within the framework of topology of integrable systems (cf. e.g. \cite{bolosh}). 
In the setting of integrable systems, the measure $\mu$ can be interpreted 
as the \textit{action variable} for the integrable Hamiltonian vector field $X_F$. Indeed, let $e$ be an edge of $\Gamma_F$, and let $\rm{Cyl} = \pi^{-1}(e)$. Choose a $1$-form $\alpha$ on $\rm{Cyl}$ such that $\diff \alpha = \omega$. Let $x_0$ and $x$ be interior points of $e$. Then, by the Stokes formula, we obtain
$$
\mu([x_0 , x]) = I(x) - I(x_0)
$$
where
$$
I(x) = \int\limits_{\mathclap{\pi^{-1}(x)}}\,\alpha.
$$
The latter expression, up to a factor $2\pi$,  is the Arnold--Mineur formula for the action (see~\cite{Arnold}).
}
\end{remark}

The above  properties of the measure $\mu$ on the graph $\Gamma_F$ 
make it natural to introduce the following definition of an abstract Reeb graph with measure.

\begin{definition}\label{MRGDef}
{\rm
A \textit{measured Reeb graph} $(\Gamma,  f, \mu)$ is an oriented connected graph  $\Gamma$ with a continuous function $f \colon \Gamma \to \R$ and a measure $\mu$ which satisfy the following properties.
\begin{longenum}
\item All vertices of $\Gamma$ are either $1$-valent or $3$-valent. 
\item For each $3$-valent vertex, there are either two incoming and one outgoing edge, or vice versa.
\item The function $f$ is strictly monotonous on each edge of $\Gamma$, and the edges of $\Gamma$ are oriented towards the direction of increasing $f$.
\item The measure $\mu$ satisfies the properties listed in Proposition \ref{measureproperty}.
\end{longenum}
Two measured Reeb graphs are \textit{isomorphic} if the graphs are homeomorphic as topological spaces, and the homeomorphism between them preserves the function and the measure.
Whenever there is no ambiguity, we denote the  measured Reeb graph  $(\Gamma,  f, \mu)$
by its graph notation $\Gamma$ only.
}
\end{definition}

\begin{definition}\label{def:compatible}
{\rm
A measured Reeb graph  $(\Gamma,  f, \mu)$ is \textit{compatible with} 
a symplectic surface $(M, \omega)$ if $\dim \Hom_1(\Gamma, \R)$ is equal to the genus 
of  $M$, and the volume of $\Gamma$ with respect to the measure $\mu$ 
is equal to the volume of $M$: $\int_\Gamma \mu=\int_M\omega$.
}
\end{definition}

Thus to each simple Morse function $F$ on a symplectic surface $(M, \omega)$, we associate a measured Reeb graph $(\Gamma_F,  f,  \mu)$ compatible with $M$. Clearly, this graph does not change, i.e. is invariant, under the action of $\SDiff(M)$ on simple Morse functions. In what follows, we show that this invariant is complete.


\medskip

\subsection{Classification of simple Morse functions under the  $\SDiff(M)$ action}
\label{sdiff-function-classif}

\begin{theorem}\label{thm1}\label{thm:sdiff-functions}
Let $M$ be a closed connected symplectic surface. Then there is a one-to-one correspondence between simple Morse functions on $M$, considered up to symplectomorphism, and (isomorphism classes of) measured Reeb graphs compatible with $M$. In other words, the following statements hold.
\begin{longenum}
\item Let $F,G \colon M \to \R$ be two simple Morse functions. 
Then the following two condition are equivalent:
\begin{longenum} \item there exists a symplectomorphism $\Phi \colon M \to M$ such that $\Phi_* F = G$;
\item the measured Reeb graphs $\Gamma_F$ and $\Gamma_G$ are isomorphic. \end{longenum}
 Moreover, any isomorphism $\phi \colon \Gamma_F \to \Gamma_G $ can be lifted to a symplectomorphism $\Phi \colon M \to M$ such that $\Phi_*F = G$.
\item For each measured Reeb graph $(\Gamma,  f,  \mu)$ compatible with $M$, there exists a simple Morse function $F \colon M \to \R$ such that its measured Reeb graph  $\Gamma_F$
coincides with $(\Gamma,  f,  \mu)$ .
\end{longenum}
\end{theorem}

This theorem gives a complete classification of simple Morse functions on a closed symplectic
surface $M$ with respect to the  $\SDiff(M)$-action in terms of their measured Reeb graphs. 
Assume now that $M$ is a 2-dimensional sphere $S^2$. In this case,  the group $\SDiff(S^2)$ is connected, i.e. coincides with $\SDiff_0(S^2)$, and moreover all symplectomorphisms are Hamiltonian, i.e., 
$\SDiff(S^2)=\SDiff_0(S^2)={\rm Ham}(S^2)$. Note also that the Reeb graph of any Morse function 
on $S^2$ has no cycles, and hence it is a tree. In this case Theorem  \ref{thm1}  allows one to completely classify  simple Morse functions on $S^2$ with respect to actions of both Hamiltonian 
diffeomorphisms and symplectomorphisms.

\begin{corollary}\label{cor:sdiff-sphere}
Let M be a 2-dimensional sphere $S^2$. Then there is a one-to-one correspondence between simple Morse functions on M, considered  either up to symplectomorphisms or up to Hamiltonian diffeomorphisms, and measured Reeb graphs compatible with $S^2$, i.e., acyclic measured Reeb graphs of total measure equal to the symplectic area of $S^2$. 
\end{corollary}

Before we prove Theorem \ref{thm1}, we need to discuss symplectic invariants of \textit {simple Morse fibrations}. Roughly speaking, a \textit{simple Morse fibration} is a fibration locally given by level sets of a simple Morse function. A precise definition is as follows.

\begin{definition}
{\rm
Let $(M, \omega)$ be a closed connected symplectic surface, and let $\Fibr$ be a (possibly singular) fibration on $M$.
Let also $\Gamma_{\Fibr}$ be the base of $\Fibr$, and let $\pi \colon M \to \Gamma_{\Fibr}$ be the projection map. Then $\Fibr$ is a called a \textit{simple Morse fibration} if for each $x \in \Gamma_{\Fibr}$ there exists its neighborhood $U(x)$ and a simple Morse function $F \colon \pi^{-1}(U(x)) \to \R$ such that for each $y \in U(x)$ the fiber $\pi^{-1}(y)$ is a connected component of a level set of $F$.
}
\end{definition}

Invariants of simple Morse fibrations under symplectic diffeomorphisms were described by Dufour, 
Molino, and Toulet \cite{DTM}. Let us recall the construction of these invariants. The first invariant 
is the Reeb graph, i.e. the base $\Gamma_{\Fibr}$ of the fibration $\Fibr$. Note that the Reeb 
graph associated with a fibration {\it does not} have a natural function on it and hence the graph
 is no longer parametrized and oriented. However the notions of a trunk and a branch for a given vertex still make sense.\par
 
Further, we can construct a measure on $\Gamma_{\Fibr}$ in the same way as in Section \ref{MRG}. On each of the edges of $\Gamma_{\Fibr}$, this measure is completely characterized by the total length of this edge. Apart from this, there are certain invariants associated to each of the $3$-valent vertices. To describe these invariants, we need the following stronger version of Lemma \ref{MDL}.

\begin{lemma}[Dufour, Molino, and Toulet \cite{DTM}]\label{MDL2}
Assume that $M$ is a closed connected symplectic surface, and let $F \colon M \to \R$ be a smooth function. Let also $O$ be a hyperbolic Morse critical point of the function $F$. Then the following statements hold.
\begin{longenum} 
\item There exists a chart $(p,q)$ centered at $O$ such that $\omega = \diff p \wedge \diff q$, and $F = \lambda \circ S $ where $S = pq$. The function $\lambda$ is smooth in a neighborhood of the origin, and $\lambda'(0) \neq 0$. Moreover, the chart $(p,q)$ can be chosen in such a way that the constant $\eps_0$ entering expansion  \eqref{DMTFormulas} is equal to $+2$.
\item If $(p',q')$ is another chart with the same properties, then $p'q' = pq$ modulo a function flat at the point $O$.
\end{longenum}
\end{lemma}

\par
\begin{remark}
{\rm
Let $O$ be a hyperbolic (i.e., saddle) singular point of the fibration $\Fibr$ on $M$. Since the fibration $\Fibr$ is locally given by level sets of a simple Morse function $F$, we can apply Lemma~\ref{MDL2}. It follows from the second statement of the lemma that the function ${S} = pq$ is well defined up to a flat function. In particular, its Taylor expansion at the point $O$ does not depend on the choice of $F$. This implies that the Taylor expansions of the functions $\zeta_0(z), \zeta_1(z),$ and  $\zeta_2(z)$ entering the expansion \eqref{DMTFormulas} are well-defined symplectic invariants of the fibration $\Fibr$ at a trivalent point. 
Let us denote these Taylor power series by $[\zeta_0],[\zeta_1],[\zeta_2]$.
}
\end{remark}
 \begin{figure}[t]
\centerline{
\begin{tikzpicture}[thick]
    \node at (3.5,3.4) (nodeA) {$\ell(e_i)$};
        \node at (5.3,1) (nodeB) {$[\zeta_0], [\zeta_1], [\zeta_2](v_k)$};
    \node [vertex] at (3,0) (nodeC) {};
    \node [vertex]  at (3,1.05) (nodeD) {};
    \node [vertex] at (3,2.75) (nodeE) {};
    \node [vertex]  at (3,4) (nodeF) {};
    \node  at (3,1.05) (nodeDdouble) {};
    \draw   (nodeC) -- (nodeD);
    \fill (nodeC) circle [radius=1.5pt];
    \fill (nodeD) circle [radius=1.5pt];
    \fill (nodeE) circle [radius=1.5pt];
    \fill (nodeF) circle [radius=1.5pt];
    \draw   (nodeD) arc (-45:44:1.2cm);
    \draw   (nodeD) arc (225:136:1.2cm);
    \draw   (nodeE) -- (nodeF);
    \draw  [dashed, ->] (nodeB) -- (nodeDdouble);
    \node at (2,2) (nodeG) {$\Gamma_\Fibr$};
\end{tikzpicture}
}
\caption{A weighted Reeb graph $\Gamma_{\Fibr}$, whose weights are as follows: each edge $e_i$ is assigned its length $\ell(e_i)$, and each $3$-vertex $v_k$ is assigned its triple of Taylor series $[\zeta_0], [\zeta_1], [\zeta_2](v_k)$.}\label{weightedTorus}
\end{figure}
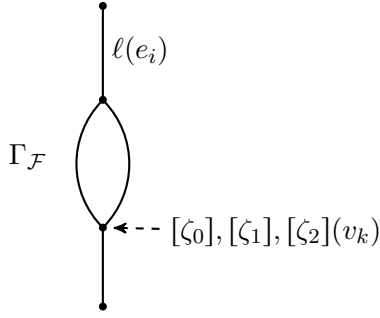
\begin{definition}\label{def:weighted}
{\rm
A \textit{weighted Reeb graph} $(\Gamma, \{\ell(e), [\zeta_j](v)\})$ is a connected graph  $\Gamma$ such that all its vertices are either $1$-valent or $3$-valent, with the following additional data (see Figure~\ref{weightedTorus}).

\begin{longenum}
\item Each edge $e$ is equipped with a positive real number - its length $\ell(e)$.
\item For each $3$-valent vertex, two of the adjacent edges are called branches, and the third is called a trunk. Each such a vertex $v$ is equipped with three real power series $[\zeta_0],[\zeta_1],[\zeta_2]$ of a variable $z$ at $0\in \R$, where $[\zeta_0]$ is assigned to the trunk, while $[\zeta_1]$ and $[\zeta_2]$ are assigned to two branches respectively, and  such that $[\zeta_0]+[\zeta_1]+[\zeta_2]= 0$. 
\end{longenum}

Note that an edge can be a branch for one of its endpoints, and a trunk for the other. 
Of course, only two of the power series, say $[\zeta_1]$ and $[\zeta_2]$, are independent, while
 $[\zeta_0]=-([\zeta_1]+[\zeta_2])$.
 
Two weighted Reeb graphs are isomorphic if they are isomorphic as combinatorial graphs, and the isomorphism between them preserves all the additional data. In particular, trunks are mapped to trunks and branches are mapped to branches.
}
\end{definition}

As explained in the remark above, the Reeb graph $\Gamma_{\Fibr}$ associated with a fibration $\Fibr$ can be naturally endowed with the structure of a weighted Reeb graph.

\begin{theorem}[Dufour, Molino, and Toulet \cite{DTM}]\label{thmDTM}
Simple Morse fibrations on closed connected symplectic surfaces are classified by weighted Reeb graphs. More precisely, the following statements hold.
\begin{longenum}
\item Let $M$ and $ N$ be two closed connected symplectic surfaces and let $\Fibr$ and $\Fibrtwo$ be simple Morse fibrations on $M$ and $N$ respectively. Then the following statements are equivalent:
\begin{longenum}
\item there exists a symplectomorphism $\Phi \colon M \to N$ which maps the fibers of $\Fibr$ to fibers of $\Fibrtwo$; \item the weighted Reeb graphs $\Gamma_{\Fibr}$ and $\Gamma_{\Fibrtwo}$ are isomorphic. 
 \end{longenum}
\item For each weighted Reeb graph $(\Gamma, \{\ell(e), [\zeta_j](v)\})$, there exists a closed connected symplectic  surface $M$, and a simple Morse fibration $\Fibr$ on $M$ such that the weighted Reeb graph $\Gamma_{\Fibr}$  of the fibration $\Fibr$ is $(\Gamma, \{\ell(e), [\zeta_j](v)\})$.\\
\end{longenum}
\end{theorem}

In what follows, we need a slightly stronger version of Theorem \ref{thmDTM}\,$i)$.

\begin{customthm}{\ref{thmDTM}$'$}\label{thmDTM2}
 Let $M$ and $ N$ be two closed connected symplectic surfaces and let $\Fibr$ and $\Fibrtwo$ be simple Morse fibrations on $M$ and $N$ respectively. 
Let also $\phi \colon \Gamma_{\Fibr} \to \Gamma_{\Fibrtwo}$ be an isomorphism of weighted Reeb graphs. Then there exists a symplectomorphism of surfaces $\Phi \colon M \to N$ which is a lift of 
the graph isomorphism $\phi$  in the following sense.

$a$) The mapping $\Phi$ maps fibers of $\Fibr$ to fibers of $\Fibrtwo$.

$b$) If $e_1$ is an edge of $\Gamma_{\Fibr} $ and $\phi(e_1) = e_2$, then the cylinder $\pi_{\Fibr}^{-1}(e_1)$ is mapped under $\Phi$ to the cylinder $\pi_{\Fibrtwo}^{-1}(e_2)$ where $\pi_{\Fibr} \colon M \to  \Gamma_{\Fibr} $ and  $\pi_{\Fibrtwo} \colon N \to  \Gamma_{\Fibrtwo} $ are natural projections. 
\end{customthm}
The proof of Theorem  \ref{thmDTM2} follows the lines of the proof of Theorem  \ref{thmDTM}\,$i)$. Details of the proof can be found in Toulet's thesis \cite{Toulet}.

\par
Now, note that each measured Reeb graph $(\Gamma,  f, \mu)$ can be viewed as a weighted Reeb graph $(\Gamma,\{\ell(e), [\zeta_j](v)\})$. Namely, assume that  $\Gamma$ is a measured Reeb graph $\Gamma_F$ associated with a simple Morse function $F$ on a symplectic manifold $M$. The function $F$ defines a fibration $\Fibr$ on $M$. Clearly, the weighted Reeb graph $\Gamma_{\Fibr}$ coincides, as an abstract graph, with the graph $\Gamma$. We claim that the structure of a measured Reeb graph on $\Gamma$ uniquely determines the structure of a weighted Reeb graph $(\Gamma, \{\ell(e), [\zeta_j](v)\})$. Obviously, the lengths $\ell(e)$ of the edges are immediately recovered from the measure, so 
it suffices to show how to recover the powers series $[\zeta_j]$ associated with $3$-valent vertices.

\begin{proposition}\label{fromMeasuredGraphToWeightedGraph}
Let $v$ be a $3$-valent vertex of $\Gamma$.
Let also $e_0$ be the trunk of $v$, and let $e_1, e_2$ be the branches of $v$.
Then $[\zeta_i](z)$ is equal to the Taylor expansion of the function
\begin{align}\label{zetaformula}
\eps_i z \ln \left|\frac{\psi^{-1}(z)}{z} \right| + \eta_i(\psi^{-1}(z))
\end{align}
at $z = 0$ where $\eps_0 = 2$,  $\eps_1 = \eps_2 =  -1$, and $\psi, \eta_0, \eta_1, \eta_2$ are defined in Proposition \ref{measureproperty}(iii).

\end{proposition}
\begin{proof}
Combining formulas \eqref{saddleAsymptotic} and  \eqref{DMTFormulas}  we get
\begin{align*}
  \eps_i s \ln |  s | + \zeta_i(  s ) = \eps_i\psi(  f )\ln |   f | + \eta_i(   f ) = \eps_i\psi(  f )\ln |  s | + \,\eps_i\psi(  f )\ln \left|  \frac{  f}{s}  \right|  + \eta_i(   f )\,.
\end{align*} 
Therefore, modulo flat functions, we have 
$$
s \equiv \psi(  f ), \quad \zeta_i(  s ) \equiv \,\eps_i\psi(  f )\ln \left|  \frac{  f}{s}  \right|  + \eta_i(   f ) .
$$
Then the statement follows by substituting $  f \equiv \psi^{-1}(s)$ into the latter equation.
\end{proof}

Thus, if we have an ``abstract'' measured Reeb graph (i.e. any object satisfying 
Definition~\ref{MRGDef}), then we can regard Proposition \ref{fromMeasuredGraphToWeightedGraph} as a definition. Therefore, any measured Reeb graph can be viewed as a weighted Reeb graph as well. 

\begin{proof}[Proof of Theorem \ref{thm1}]
Let us prove the first statement of Theorem \ref{thm1}.
The implication (a) $\Rightarrow$ (b) follows from the construction of the measured Reeb graph associated with a simple Morse function, so it suffices to prove the implication (b) $\Rightarrow$ (a). \par 

Assume that $\phi \colon \Gamma_F \to \Gamma_G $ is an isomorphism of measured 
Reeb graphs. Then it can be viewed as an isomorphism of weighted Reeb graphs. 
Therefore, by Theorem \ref{thmDTM2}, there exists a symplectomorphism $\Phi \colon M \to M$ 
which maps the fibers of $\Fibr$ to fibers of $\Fibrtwo$. Moreover, if $e_1$ is an edge 
of $\Gamma_{\Fibr} $ and $\phi(e_1) = e_2$, then the map $\Phi$ takes
the cylinder $\rm{Cyl}_{1} = \pi_{\Fibr}^{-1}(e_1)$ 
to the cylinder $\rm{Cyl}_{2} = \pi_{\Fibrtwo}^{-1}(e_2)$. 
\par
Now, let us again consider $\Gamma_F $ and $ \Gamma_G $ as measured Reeb graphs. Since the map $\Phi \colon \rm{Cyl}_{1} \to \rm{Cyl}_{2}$ is fiberwise, it descends to a map $\phi' \colon e_1 \to e_2$. Moreover, since the diffeomorphism $\Phi$ is symplectic, the graph map $\phi'$ is measure-preserving, so $\phi' = \phi$. This implies that  the symplectomorphism $\Phi$ is the lift of $\phi$, and moreover that this diffeomorphism $\Phi$ maps function $F$ to $G$, q.e.d. 
\par
Now, let us prove the second statement.
Consider the measured Reeb graph $\Gamma$ as a weighted Reeb graph. By Theorem \ref{thmDTM}\,$ii)$, there exists a symplectic surface $N$, and a fibration $\Fibr$ on $N$ such that $\Gamma_\Fibr = \Gamma$.
Since the first Betti number of $\Gamma$ is equal to the genus of $M$, the surfaces $M$ and $N$ have the same genus. Moreover, since the volume of $\Gamma$ is equal to the volume of $M$, the surfaces $M$ and $N$ are of the same area. Therefore, by Moser's theorem \cite{Moser65},
surfaces $M$ and $N$ are symplectomorphic. Using this symplectomorphism, we transport the fibration $\Fibr$ from $N$ to $M$. In this way we obtain a fibration $\Fibr$ on $M$ such that the measured Reeb graph of $\Fibr$ is $\Gamma$.
\par
 Note that, a priori, there is no projection $\pi \colon M \to \Gamma$, since a weighted Reeb graph is defined as a combinatorial object. However, there is an identification between cylinders separating critical fibers of $\Fibr$, and edges of $\Gamma$. As it is easy to see, there exists a unique projection $\pi \colon M \to \Gamma$ which realizes this identification and respects the measure. This means that the measure on $\Gamma$ is the push-forward of the symplectic measure on $M$ under the map $\pi$.
 \par
Then lift the function $f$ from the graph $\Gamma$ to a function $F$ on $M$ by means of the projection $\pi$. We need to check that $F$ is a smooth function on $M$. Away from hyperbolic fibers, this follows from the first two statements of Proposition \ref{measureproperty} (recall that these statements are included in the definition of a measured Reeb graph). Therefore, it suffices to prove that $F$ is smooth near each hyperbolic fiber.
\par
Let $v$ be a $3$-valent vertex of $\Gamma$.
Let also $e_0$ be the trunk of $v$, and let $e_1, e_2$ be the branches of $v$. As it follows from the construction of the weighted Reeb graph associated with $\Fibr$, there exists a smooth function $S$ defined in a neighborhood of $\pi^{-1}(v)$ such that the fibration $\Fibr$ is locally given by connected components of level sets $\{S = \const\}$, and for each $x \in e_i$ sufficiently close to $v$, we have
\begin{align}\label{measureInTermsOfS}
\mu([v,x])&= \eps_i s(x) \ln |  s(x) | + \zeta_i(  s(x) )
\end{align}
where $\eps_0 = 2$,  $\eps_1 = \eps_2 =  -1$, and $s$ is a function defined in the neighborhood of $v$ by descending the function $S$. By definition of a weighted Reeb graph, the Taylor expansion of  $\zeta_i(z)$ at $z = 0$ coincides with the power series $[\zeta_i]$ associated with the vertex $v$. On the other hand, since the graph $\Gamma$ is actually a measured Reeb graph, the power series $[\zeta_i]$ is equal to the Taylor expansion of
\begin{align}\label{wavezeta}
\wave \zeta_i(z) = \eps_i z \ln \left|\frac{\psi^{-1}(z)}{z} \right| + \eta_i(\psi^{-1}(z))
\end{align}
at $z = 0$. Therefore, $\zeta_i \equiv \wave \zeta_i$ modulo a flat function. Now, we need the following technical lemma.
\begin{lemma}\label{flatTerm}
Let $a(z), b(z)$ be two functions which are defined and smooth in a punctured neighborhood of the origin $0\in \R$. Assume that the difference $a(z)-b(z)$ is a smooth function flat at the origin, and that $b'(z)$ is bounded away from zero.
Then there exists a diffeomorphism of the form 
$$
h(z) = z + \mbox{flat function},
$$
defined in a (possibly smaller) neighborhood of the origin, such that $b$ is obtained from $a$ by the diffeomorphism $h$:  
$
b = a \circ h.
$
\end{lemma}

\begin{remark}
{\rm
The statement of the lemma is also true if $b'(z)$ is not bounded away from zero, but there exists an integer $n$ such that $b'(z)/z^n$ is bounded away from zero in a punctured neighborhood of the origin.
}
\end{remark}

\begin{proof}[Proof of Lemma \ref{flatTerm}]
Apply the Moser path method. Instead of looking for one diffeomorphism $h$, we will be looking for a family of diffeomorphisms $h(z,t)$ such that
\begin{align}\label{homotopyEqn}
b(z) = t \cdot a(h(z,t)) + (1 - t)\cdot b(h(z,t))
\end{align}
and
$h(z, 0) = z$. 
Differentiating \eqref{homotopyEqn} with respect to $t$, we obtain the following differential equation
\begin{align}\label{homologicalEqn}
\diffFXp{h}{t} = \frac{b(h) - a(h)}{b'(h) + t\cdot(a'(h) - b'(h))}.
\end{align}
Using that $b'$ is bounded away from zero, we conclude that the right-hand side of \eqref{homologicalEqn} is flat in $h$ for any fixed $t$. This easily implies that if $z$ is sufficiently small, the solution of \eqref{homologicalEqn} with initial condition $h(z, 0) = z$ is extendable up to time $t = 1$ and has the form
$$
h(z,t) = z + r(z,t)
$$
where $r(z,t)$ is flat in $z$ for all $t \in [0,1]$, and $r(z,0) = 0$. Finally, note that equation \eqref{homologicalEqn} together with the condition $h(z,0) = z$ imply \eqref{homotopyEqn}, therefore $b(z) = a(h(z,1))$, q.e.d.
\end{proof}
Now, we use Lemma \ref{flatTerm} to find functions $h_i(s)$ such that $h_i(s) - s $ is flat at $s=0$, and
\begin{align}\label{changeOfS}
\eps_i s \ln |  s | + \zeta_i(  s ) = \eps_i h_i(s) \ln |  h_i(s) | + \wave \zeta_i(  h_i(s) ).
\end{align}
Combining equations \eqref{saddleAsymptotic},\eqref{measureInTermsOfS}, and \eqref{changeOfS}, for a function $f$ normalized by the condition $f(v)=0$ we obtain 
$$
\eps_i\psi(  f )\ln |   f | + \eta_i(   f ) =  \eps_i h_i(s) \ln |  h_i(s) | + \wave \zeta_i(  h_i(s) ).
$$
Using \eqref{wavezeta}, we conclude that
$$
\eps_i h_i(s) \ln |\psi^{-1}(h_i(s))| + \eta_i(\psi^{-1}(h_i(s))) = \eps_i\psi(  f )\ln |   f | + \eta_i(   f ), 
$$
and thus
$$
f(x) = \psi^{-1}(h_i(s(x)))
$$
for any $x \in e_i$ sufficiently close to $v$. 
Therefore,  on $\pi^{-1}(e_i) \subset M$, one has
$$
F = \psi^{-1} \circ h_i \circ S.
$$
Hence, since $h_i(z) - z$ is flat at $z = 0$, we conclude that $F$ is a smooth function near
a hyperbolic level and hence everywhere.
It is also easy to see that $F$ is a simple Morse function (since so is the function $S$), and that its measured Reeb graph coincides with $\Gamma$, as desired. Theorem  \ref{thm1} is proved.
\end{proof}

As we mentioned, Theorem \ref{thm:sdiff-functions}, in addition to classifying generic functions 
with respect to $\SDiff(M)$-action on any surface $M$, 
also describes the $\SDiff_0(M)$-classification in the case of $M=S^2$.
In the general case of $M$ of an arbitrary genus the classification of functions 
with respect to the $\SDiff_0(M)$-action, i.e. by symplectomorphisms isotoped to the identity,
is much more subtle than that for $\SDiff(M)$-action.
Now we describe this classification, i.e.,  the discrete invariants, for the $\SDiff_0(M)$-action 
on simple Morse functions.  It turns out it is convenient to treat separately the cases of $genus(M)=1$ and $genus(M)\ge 2$.


\bigskip

\subsection{Classification of simple Morse functions under the $\SDiff_0(M)$ action: genus one case}\label{sect:genus-one}

Assume that $M = \T^2$ is a symplectic two-dimensional torus, a symplectic surface of genus one with a fixed symplectic form $\omega$, and let $F \colon \T^2 \to \R$ be a simple Morse function on $\T^2$. The projection $\pi \colon \T^2 \to \Gamma_F$ from $\T^2$ to the Reeb graph $\Gamma_F$ of $F$ induces an epimorphism
$$
\pi_* \colon \Hom_1(\T^2,\Z) \to \Hom_1(\Gamma_F, \Z).
$$

\begin{definition} \label{def:freeze-torus}
{\rm
Let  $\T^2$
be a symplectic two-dimensional  torus. A measured Reeb graph $\Gamma$ compatible 
with $\T^2$ is \textit{frozen into} $\T^2$ if it is endowed with an epimorphism 
$\pi_* \colon \Hom_1(\T^2,\Z) \to \Hom_1(\Gamma, \Z)$. 
Two measured Reeb graphs $(\Gamma_1, (\pi_1)_{*})$ and $(\Gamma_2, (\pi_2)_{*})$ 
frozen into the same 
torus $\T^2$ are isomorphic if there exists an isomorphism $\phi \colon \Gamma_1 \to \Gamma_2$
 of measured Reeb graphs such that the following diagram commutes
\begin{equation}\label{comDiag}
\centering
\begin{tikzcd}
& \Hom_1(\T^2,\Z) \arrow{dl}[swap]{(\pi_1)_{*}}  \arrow{dr}{(\pi_2)_{*}}&\\
 \Hom_1(\Gamma_1,  \Z)  \arrow{rr}{\phi_*} & & \Hom_1(\Gamma_2,  \Z)\,,
\end{tikzcd}
\end{equation}
where $(\pi_1)_*$ is the freezing homomorphism 
of $\Gamma_1$, and  $(\pi_2)_*$ is the freezing homomorphism of $\Gamma_2$.
}
\end{definition}

Thus, to each simple Morse function $F$ on a symplectic two-dimensional torus $(\T^2, \omega)$, we associate a measured Reeb graph $(\Gamma_F, (\pi_F)_{*})$ frozen into $\T^2$. This graph is invariant under the action of $\SDiff_0(\T^2)$ on simple Morse functions. The following theorem states that this invariant is complete.

\begin{theorem}\label{Sdiff0Torus}
Let $(\T^2, \omega)$ be a symplectic two-dimensional torus. Then there is a one-to-one correspondence between simple Morse functions on $\T^2$, considered up to symplectomorphism isotopic to the identity, and (isomorphism classes of) measured Reeb graphs frozen into~$\T^2$. In other words, the following statements hold.
\begin{longenum}
\item Let $F,G \colon \T^2 \to \R$ be two simple Morse functions. Then the following two conditions are equivalent:
\begin{longenum} \item there exists a symplectomorphism $\Phi \colon \T^2 \to \T^2$ isotopic to the identity such that $\Phi_* F = G$;
\item the  measured Reeb graphs $(\Gamma_F, (\pi_F)_{*})$ and $(\Gamma_G, (\pi_G)_{*})$ 
frozen into $(\T^2, \omega)$ are isomorphic. \end{longenum}
 Moreover, any isomorphism $\phi \colon (\Gamma_F, (\pi_F)_{*})\to (\Gamma_G, (\pi_G)_{*})$ of  measured Reeb graphs frozen into the same torus $\T^2$
 can be lifted to a symplectomorphism $\Phi \colon \T^2 \to \T^2$ isotopic to the identity such that $\Phi_*F = G$.
\item For each measured Reeb graph $(\Gamma, \pi_{*})$ frozen into $\T^2$, there exists a simple Morse function $F \colon \T^2 \to \R$ such that the frozen measured Reeb graph  $(\Gamma_F, (\pi_F)_{*})$ of $F$ is $(\Gamma, \pi_{*})$.
\end{longenum}
\end{theorem}

\begin{proof}
Let us prove the first statement.
The implication (a) $\Rightarrow$ (b) is obvious, so it suffices to prove the implication (b) $\Rightarrow$ (a). Assume that $\phi \colon \Gamma_F \to \Gamma_G $ is an isomorphism of measured Reeb graphs frozen into $\T^2$. By Theorem \ref{thm1}, it can be lifted to a symplectomorphism $
 \Phi' \colon \T^2 \to \T^2$ such that $\Phi'_* F = G$. From diagram \eqref{comDiag} we obtain the following commutative diagram
\begin{equation}\label{comDiag2}
\centering
\begin{tikzcd}
 \Hom_1(\T^2,\Z) \arrow{rr}{\Phi'_{*}}  \arrow{dr}[swap]{(\pi_G)_*} & & \Hom_1(\T^2,\Z)    \arrow{dl}{(\pi_G)_*}\\
 & \Hom_1(\Gamma_G,  \Z)\,, &
\end{tikzcd}
\end{equation}
where $\pi_G \colon \T^2 \to \Gamma_G$ is the natural projection. Let $a \in  \Hom_1(\T^2,\Z) $ be a cycle homologous to a connected component of a regular $G$-level, and let $b \in  \Hom_1(\T^2,\Z)$ be any cycle such that $(a,b)$ is a basis of $ \Hom_1(\T^2,\Z)$. 
Then we can find a basis element $c \in  \Hom_1(\Gamma_G,  \Z) $ such that
$$
\pi_G^*(a) = 0, \quad \pi_G^*(b) = c.
$$
Taking into account that $\Phi'$ is orientation-preserving, we see from diagram \eqref{comDiag2} that 
$$
\Phi'_*(a) = a, \quad \Phi'_*(b) = b + ma
$$
where $m \in \Z$. Now, we claim that there exists a symplectomorphism $\Psi \colon \T^2 \to \T^2$ such that $\Psi_*G = G$, and
$$
\Psi_*(a) = a, \quad \Psi_*(b) = b - ma.
$$
Indeed, such a symplectomorphism can be constructed as a suitable power of the Dehn twist about any connected component of a regular $G$-level.\par
Now, set $\Phi = \Psi \circ \Phi'$. Clearly, $\Phi$ is a symplectic diffeomorphism, and $\Phi_*F = G$. Furthermore, $\Phi$ is identical in homology. For a torus, this implies that $\Phi$ is isotopic to the identity (see e.g. Theorem 2.5 of \cite{MCG}), as desired.\par
Now, prove the second statement. By Theorem \ref{thm1}, there exists a simple Morse function $F' \colon \T^2 \to \R$ such that the measured Reeb graph associated with $F'$ is $\Gamma$. A priori, the map $(\pi_{F'})_*  \colon \Hom_1(\T^2,\Z) \to \Hom_1(\Gamma, \Z)$ does not coincide with the prescribed freezing homomorphism $\pi_*  \colon\Hom_1(\T^2,\Z) \to \Hom_1(\Gamma, \Z)$. However, we may find a symplectic map $\Psi \colon \T^2 \to \T^2$ such that $\pi_* \circ \Psi_* = (\pi_{F'})_*$. Indeed, we can find an orientation preserving diffeomorphism with this property, and, by Moser's trick, there exists a symplectic diffeomorphism in each isotopy class of orientation preserving diffeomorphisms. Now, taking $F = \Psi_*F'$, we obtain a function with desired properties.
\end{proof}

\begin{remark}
{\rm
The consideration in this section works for a surface $M$ of any genus, 
and it classifies functions up to 
symplectomorphisms trivially acting in the homology of $M$. The fact that symplectomorphisms
trivial on homology must be isotopic to the identity holds only for genus not greater than one, see
\cite{MCG}. For the $\SDiff_0(M)$-classification in higher genera one needs to incorporate finer tools and we consider them in the next section.
}
\end{remark}


\medskip

\subsection{Reduced Reeb graphs and pants decompositions}
In order for Theorem \ref{Sdiff0Torus} to hold for higher genera, we need to modify the definition 
of freezing.
Let $M$ be a closed connected two-dimensional manifold of genus $\varkappa \geq 2$, and let $F \colon M \to \R$ be a simple Morse function on $M$. Let also $\Gamma_F$ be the Reeb graph of $F$. Take an edge $e \subset \Gamma_F$, and let $C(e) = \pi^{-1}(x_e)$, 
where $x_e \in e$ is any interior point (clearly, the isotopy class of $C(e)$ is independent of the choice of interior point $x_e \in e$). Note that some of the cycles $C(e)$ are isotopic to each other, and some are contractible. To keep only non-isotopic cycles and get rid of redundant ones, we make use of a construction by Hatcher and Thurston \nolinebreak \cite{HT}, which associates a \textit{pants decomposition} of $M$ to each simple Morse function on $M$. \par

\begin{definition}
{\rm
Let $ \Gamma'_F$ be the maximal subgraph of $\Gamma_F$ with no $1$-valent vertices (equivalently, the minimal subgraph to which $\Gamma_F$ retracts).  The \textit{reduced Reeb graph} $ \bar{\Gamma}_F$ is defined by disregarding all bivalent vertices of $  \Gamma'_F $ (see Figure \ref{pretzel0}).
 There is a natural projection $M \to \bar \Gamma_F$ constructed as follows. The graph $\Gamma_F$ can be presented as 
$$\Gamma_F =  \Gamma'_F \cup T_1 \cup \dots \cup T_m$$
where $T_1, \dots, T_m \subset \Gamma_F$ are pairwise disjoint trees, $T_i \cap   \Gamma'_F = \{v_i\}$, and $v_1, \dots, v_m$ are bivalent vertices of \nolinebreak $ \Gamma'_F$. 
The mapping $r \colon \Gamma_F \to \Gamma'_F \simeq  \bar \Gamma_F$ that is identical on  $  \Gamma'_F $ and maps $T_i$ to $v_i$ is a deformation retraction.
A projection $M \to \bar \Gamma_F$ is defined by composing the projection  $\pi \colon M \to \Gamma_F$ with the retraction mapping \nolinebreak $r$. We shall denote the projection $M \to \bar \Gamma_F$  by the same letter $\pi$.
}
\end{definition}
\par
Let $e$ be an edge of $  \bar \Gamma_F$, and let $x \in e$ be its interior point. We say that $x$ is regular if it is not a bivalent vertex of $  \Gamma'_F $.

\begin{proposition}
The reduced Reeb graph $  \bar \Gamma_F$  has the following properties.
\begin{longenum}
\item  $  \bar \Gamma_F$ is homotopy equivalent to $  \Gamma_F $.
\item All vertices of $  \bar \Gamma_F$ are $3$-valent. The number of vertices is $2\varkappa - 2$, and the number of edges is $3\varkappa - 3$ where $\varkappa$ is the genus of $M$.
\item Let $e$ be an edge of $  \bar \Gamma_F$. Then for all regular interior points $x_e \in e$, the set $\pi^{-1}(x_e) \subset M$ is an embedded circle. Moreover, the isotopy class of $\pi^{-1}(x_e)$ is non-trivial and does not depend on the choice of a regular $x_e \in e$. 
\item Let $e_1,e_2$ be two distinct edges of $  \bar \Gamma_F$, and let $x_{1} \in e_1$ and $x_{2}\in e_2$ be regular interior points. Then the isotopy classes of $\pi^{-1}(x_{1})$ and $\pi^{-1}(x_{2})$ are distinct.
\end{longenum}
\end{proposition}
\begin{proof}
The proof is straightforward.
\end{proof}
  \begin{figure}[t]
\centerline{
\begin{tikzpicture}[thick, scale=1.2, rotate = -90]
\draw [->] (4,0.8) -- (2.2,0.8);
\draw (2,2.7)  .. controls (2.6, 2.7) and (3, 2.3)  ..  (3.2,2.3);
\draw (2,1.3)  .. controls (2.6, 1.3) and (3, 1.7)  ..  (3.2,1.7);
\draw (2,2.7)  .. controls (0.8,2.7) and (0.8, 1.3)  ..  (2,1.3);
\draw   (1.63,2) arc (135:45:0.7cm);
\draw   (1.63,2) arc (-135:-35:0.7cm);
\draw     (1.63,2) arc (60:120:0.53cm);
\draw  [densely dotted](1.63,2) arc (-60:-120:0.53cm);
\draw [densely dotted]   (3.2,2.3) .. controls (3.1, 2.1) and  (3.1, 1.9) ..  (3.2,1.7);
\draw  (3.2,2.3) .. controls (3.3, 2.1) and  (3.3, 1.9) ..  (3.2,1.7);
\draw (4.4,2.7)  .. controls (3.8, 2.7) and (3.4, 2.3)  ..  (3.2,2.3);
\draw (4.4,1.3)  .. controls (3.8, 1.3) and (3.4, 1.7)  ..  (3.2,1.7);
\draw (4.4,2.7)  .. controls (5.6,2.7) and (5.6, 1.3)  ..  (4.4,1.3);
\draw   (4.77,2) arc (45:135:0.7cm);
\draw   (4.77,2) arc (-45:-145:0.7cm);
\draw     (4.77,2) arc (120:60:0.53cm);
\draw  [densely dotted](4.77,2) arc (-120:-60:0.53cm);
\node [vertex] at (2.6,6) (nodeB) {};
\node [vertex] at (3.7,6) (nodeC) {};
\node [vertex] at (2.6,4.2) (nodeA) {};
\node [vertex] at (3.7,4.2) (nodeD) {};
\node [vertex] at (1.6,4.2) (nodeE) {};
\node [vertex] at (4.7,4.2) (nodeF) {};
\node [vertex] at (1.1,4.2) (nodeH) {};
\node [vertex] at (5.2,4.2) (nodeK) {};
\draw (nodeB)  .. controls +(-1.3,-0.7) and +(-1.3,0.7) ..  (nodeB);
\draw (nodeC)  .. controls +(1.3,-0.7) and +(1.3,0.7) ..  (nodeC);
\draw[->-] (nodeA)  .. controls +(-0.2,-0.3) and +(0.2,-0.3) ..  (nodeE);
\draw[->-] (nodeA)  .. controls +(-0.2,0.3) and +(0.2,0.3) ..  (nodeE);
\draw[-<-] (nodeD)  .. controls +(0.2,-0.3) and +(-0.2,-0.3) ..  (nodeF);
\draw[-<-] (nodeD)  .. controls +(0.2,0.3) and +(-0.2,0.3) ..  (nodeF);
\draw (nodeB) -- (nodeC);
\draw [->-](nodeD) -- (nodeA);
\draw [->-](nodeK) -- (nodeF);
\draw [->-](nodeE) -- (nodeH);
\node at (1.5,2.3)(nodeG) { \small $C_1$};
\node at (3.2,2.6)(nodeI) { \small $C_2$};
\node at (4.9,2.3)(nodeJ) { \small $C_3$};
\fill (nodeB) circle [radius=1.5pt];
\fill (nodeC) circle [radius=1.5pt];
\fill (nodeA) circle [radius=1.5pt];
\fill (nodeD) circle [radius=1.5pt];
\fill (nodeE) circle [radius=1.5pt];
\fill (nodeF) circle [radius=1.5pt];
\fill (nodeH) circle [radius=1.5pt];
\fill (nodeK) circle [radius=1.5pt];
\node at (4.6,3) (nodeL) {$M$};
\node at (3.1,0.6) (nodeL) {$F$};
\node at (3.1,4.6) (nodeL) {$\Gamma_F$};
\node at (3.1,6.4) (nodeL) {$\bar \Gamma_F$};
\end{tikzpicture}
}
\caption{Reeb graph, reduced Reeb graph and pants decomposition for a height function on a pretzel.}\label{pretzel0}
\end{figure}
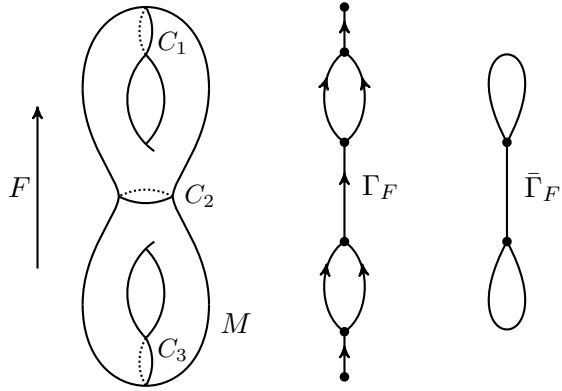
This way we obtain a collection $C_1, \dots, C_{3\varkappa-3}$  of pairwise disjoint nontrivial distinct isotopy classes of simple closed curves. It is well known that such cycles decompose $M$ into $2\varkappa-2$ manifolds with boundary $P_1$, $\dots$, $P_{2\varkappa-2}$, and that each $P_i$ is homeomorphic to a pair of pants, i.e. a sphere with three holes. For this reason, a choice of  $3\varkappa-3$ nontrivial pairwise disjoint distinct isotopy classes of simple closed curves is called a \textit{pants decomposition} of $M$. Pants decomposition are also known as \textit{maximal cut systems}. 
Thus, to each simple Morse function $F$ on $M$ we associate a pants decomposition $\pazocal P_F$ of $M$. 
\begin{example}
{\rm
Figure \ref{pretzel0} shows a height function on a pretzel, as well as its Reeb graph, reduced Reeb graph and the associated pants decomposition.
}
\end{example}

For each pants decomposition $\pazocal P$, there is an associated graph $\Gamma(\pazocal P) $. This graph is defined as follows: the vertices of this graph are pairs of pants  $P_1, \dots, P_{2\varkappa-2}$. Two vertices $v_i$ and $v_j$ are joined by an edge if the pairs  of pants $P_i$ and $P_j$ have a common boundary component.
%
%
  \begin{figure}[t]
\centerline{
\begin{tikzpicture}[thick, scale=0.8]
\draw   (1.68,3.2) arc (60:300:1.46cm);
\draw   (2.32,3.2) arc (120:-120:1.46cm);
\draw (1.68,3.2)  .. controls (1.9, 3.1) and (2.1, 3.1) ..  (2.32,3.2) ;
\draw (1.68,0.67)  .. controls (1.9, 0.77) and (2.1, 0.77) ..  (2.32,0.67) ;
    \draw   (1,1.3) arc (260:100:0.6cm);
    \draw   (0.8,1.4) arc (-80:80:0.5cm);
        \draw   (3.25,1.3) arc (260:100:0.6cm);
    \draw   (3.05,1.4) arc (-80:80:0.5cm);
\draw    [densely dotted] (2.75, 1.95) arc (65:115:1.85cm);
\draw   (2.75, 1.95) arc (-65:-115:1.85cm);
\draw   [densely dotted] (4.5, 1.95) arc (60:120:1cm);
\draw (4.5, 1.95) arc (-60:-120:1cm);
\draw   [densely dotted]  (0.5, 1.95) arc (60:120:1cm);
\draw  (0.5, 1.95) arc (-60:-120:1cm);
\node [vertex]  at (8,1) (nodeD) {};
\node [vertex] at (8,3) (nodeE) {};
\draw   (nodeD) -- (nodeE);
\fill (nodeD) circle [radius=1.5pt];
\fill (nodeE) circle [radius=1.5pt];
\draw   (nodeD) arc (-45:45:1.44cm);
\draw   (nodeD) arc (225:135:1.44cm);
\end{tikzpicture}
}
\caption{A pants decomposition $\pazocal P$ of a pretzel and the associated graph  
$\Gamma(\pazocal P)$.}\label{pretzel}
\end{figure}
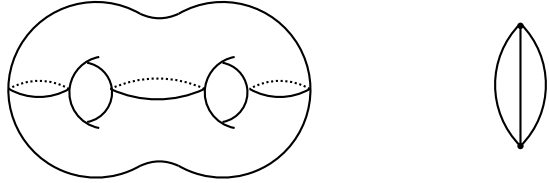
In other words, there exist bijections
\begin{align*}
C \colon \{ \mbox{edges of } \Gamma(\pazocal P)\} \to \{ \mbox{cycles of }\pazocal P  \},  \quad P \colon \{ \mbox{vertices of } \Gamma(\pazocal P)\} \to \{ \mbox{pairs of pants of }\pazocal P  \} \,, 
\end{align*}
such that a vertex $v$ of the graph $ \Gamma(\pazocal P)$ is adjacent to an edge $e$ if and only if the circle $C(e)$ is a boundary component of the pair of pants $ P(v)$. Note that if the graph $ \Gamma(\pazocal P)$ admits non-trivial automorphisms, then there exist different maps $C,P$ with these properties. This motivates us to give the following definition.
\begin{definition}\label{color}
{\rm
A {\it colored pants decomposition} of a surface $M$ is a quadruple $(\pazocal P, \Gamma, C, P)$ where $\pazocal P$ is a pants decomposition of $M$, $\Gamma$ is a $3$-valent graph, and $C,P$ are bijections
\begin{align*}
C \colon \{ \mbox{edges of } \Gamma(\pazocal P)\} \to \{ \mbox{cycles of }\pazocal P  \},  \quad P \colon \{ \mbox{vertices of } \Gamma(\pazocal P)\} \to \{ \mbox{pairs of pants of }\pazocal P  \} \,, 
\end{align*}
such that a vertex $v$ of the graph $\Gamma$ is adjacent to an edge $e$ if and only if the circle $C(e)$ is a boundary component of the pair of pants $ P(v)$.
}
\par
Two colored pants decompositions  $(\pazocal P_1, \Gamma_1, C_1, P_1)$ and $(\pazocal P_2, \Gamma_2, C_2, P_2)$ are isomorphic if the pants decompositions coincide 
($\pazocal P_1=\pazocal P_2$) and 
there exists an isomorphism of graphs $\phi \colon \Gamma_1 \to \Gamma_2$, such that $P_1 = P_2 \circ \phi $, and $C_1 = C_2 \circ \phi $. Note that the map $C$ entering this definition
 uniquely determines the map $P$ with the only exception: unless $\pazocal P$ and $\Gamma$ are the ones depicted in Figure \ref{pretzel}.
\end{definition}

Clearly, any pants decomposition $\pazocal P$ can be viewed as a pants decomposition 
colored by its graph $\Gamma(\pazocal P)$. Therefore, when we say that a pants decomposition 
$\pazocal P$ is colored by a graph $\Gamma$, this means that $\Gamma$ is isomorphic to 
$\Gamma(\pazocal P)$, and that the isomorphism between $\Gamma$ and 
$\Gamma(\pazocal P)$ is fixed.
\par
 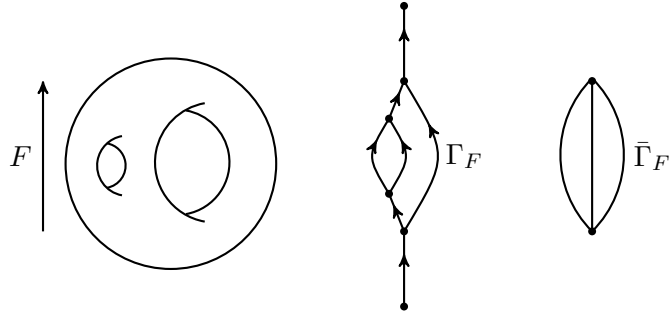
\begin{figure}[t]
\centerline{
\begin{tikzpicture}[thick]
    \draw (2.4,1.9) ellipse (1.4cm and 1.4cm);
    \draw   (2.85,1.13) arc (260:100:0.8cm);
    \draw   (2.6,1.23) arc (-80:80:0.7cm);
    \draw   (1.75,1.48) arc (260:100:0.4cm);
    \draw   (1.55,1.58) arc (-80:80:0.3cm);
    %
    \draw  [->] (0.7,1) -- (0.7,3);
    \node at (0.4,2) (nodeA) {$F$};
    \node at (6.3,2) (nodeB) {$\Gamma_F$};
    \node [vertex] at (5.5,0) (nodeC) {};
    \node [vertex]  at (5.5,1) (nodeD) {};
    \node [vertex] at (5.5,3) (nodeE) {};
    \node [vertex]  at (5.5,4) (nodeF) {};
    \node [vertex]  at (5.3,1.5) (nodeG) {};    
        \node [vertex]  at (5.3,2.5) (nodeH) {};    
    \draw  [->-] (nodeC) -- (nodeD);
    \fill (nodeC) circle [radius=1.5pt];
    \fill (nodeD) circle [radius=1.5pt];
    \fill (nodeE) circle [radius=1.5pt];
    \fill (nodeF) circle [radius=1.5pt];
        \fill (nodeG) circle [radius=1.5pt];
            \fill (nodeH) circle [radius=1.5pt];
    \draw  [->-] (nodeE) -- (nodeF);
    \draw[->-] (nodeG)  .. controls +(-0.3,+0.5) ..  (nodeH);
        \draw[->-] (nodeG)  .. controls +(0.3,+0.5) ..  (nodeH);
                \draw[->-] (nodeD) -- (nodeG);
        \draw[->-] (nodeH) -- (nodeE);
     \draw[->-] (nodeD)  .. controls +(0.6,+1) ..  (nodeE);       
\node [vertex]  at (8,1) (nodeI) {};
\node [vertex] at (8,3) (nodeJ) {};
\draw   (nodeI) -- (nodeJ);
\fill (nodeI) circle [radius=1.5pt];
\fill (nodeJ) circle [radius=1.5pt];
\draw   (nodeI) arc (-45:45:1.44cm);
\draw   (nodeI) arc (225:135:1.44cm);
    \node at (8.8,2) (nodeK) {$\bar \Gamma_F$};
\end{tikzpicture}
}
\caption{A height function on a pretzel whose reduced Reeb graph has no simple loops.}\label{pretzel2}
\end{figure}
Clearly, the graph of the pants decomposition $\pazocal{P}_F$ associated with a Morse function $F$ has a natural structure of a pants decomposition colored by the reduced Reeb graph  $  \bar \Gamma_F$.
The colored pants decomposition $(\pazocal P_F, \bar \Gamma_F)$, together with the measured Reeb graph $\Gamma_F$ is invariant under the $\SDiff_0(M)$-action on simple Morse functions
on $M$. If the reduced Reeb graph $\bar \Gamma_F$ has no simple loops (see e.g. Figure \ref{pretzel2}), then it turns out that there are no other invariants. Otherwise, there are additional invariants associated to each of the loops. These invariants are constructed as follows.\par

  \begin{figure}[t]
\centerline{
\begin{tikzpicture}[thick, scale=1.5, rotate = -90]
\draw (2,2.7)  .. controls (2.6, 2.7) and (3, 2.2)  ..  (3.2,2.3);
\draw (2,1.3)  .. controls (2.6, 1.3) and (3, 1.8)  ..  (3.2,1.7);
\draw (2,2.7)  .. controls (0.8,2.7) and (0.8, 1.3)  ..  (2,1.3);
\draw   (1.63,2) arc (135:45:0.7cm);
\draw   (1.63,2) arc (-135:-35:0.7cm);
\draw     (1.63,2) arc (60:120:0.53cm);
\draw  [densely dotted](1.63,2) arc (-60:-120:0.53cm);
\draw   (3.2,2.3) .. controls (3.3, 2.1) and  (3.3, 1.9) ..  (3.2,1.7);
\draw   (3.2,2.3) .. controls (3.1, 2.1) and  (3.1, 1.9) ..  (3.2,1.7);
\node  at (2.3,3.7) (nodeF) {\Huge $\rightarrow$};
\node  at (2.5,3.7) (nodeD) { $\pi$};
\node [vertex] at (2.5,5) (nodeB) {};
\node [vertex] at (3,5) (nodeC) {};
\draw [a] (nodeB)  .. controls +(-1.2,-1.5) and +(-1.2,1.5) ..  (nodeB);
\draw  (nodeC) -- (nodeB);
\node at (2.35,5) (nodeC) {$v$};
\node at (2,5.5) (nodeE) {$e$};
\node at (1,2.3)(nodeG) { \small $C(e)$};
\node at (2.8,5.15) (nodeH) {$e'$};
\node at (3.2,2.6)(nodeI) { \small $C(e')$};
\node at (2.5,1.1)(nodeJ) { \small $P(v)$};
\fill (nodeB) circle [radius=1.5pt];
\draw [->]  (1.35,2.06) -- (1.35, 2.35) {};
\end{tikzpicture}
}
\caption{Pair of pants corresponding to a loop.}\label{loop}
\end{figure}
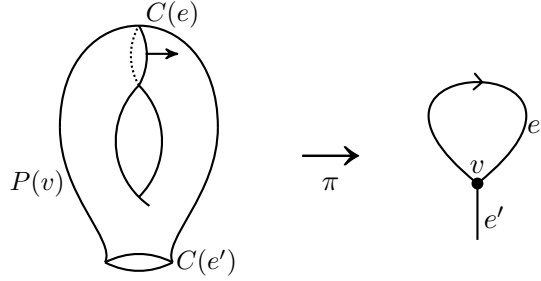


Assume that $ e$ is a loop in $\bar \Gamma_F$, 
i.e., an edge joining some vertex $v$ with itself. Then the pair of pants $P(v)$ is embedded into the surface $M$ as depicted in Figure \ref{loop}.  Choose an arbitrary orientation of the loop $e$. Then, using the projection $\pi \colon M \to \bar \Gamma_F$, one can lift this orientation to a coorientation of the cycle $C(e)$ (see Figure \ref{loop}). Since the surface $M$ is oriented, a coorientation of the cycle $C(e)$ canonically defines an orientation of this cycle. This way, we obtain a bijective mapping
$$
\mathrm{ht_e} \colon \{ \mbox{orientations of }e\} \to  \left\{\mbox{orientations of } C(e)\right\}.
$$
Since there are two such bijections, the invariant $\halfTwist_e$ can take two values.
\begin{definition}
The map $\halfTwist_e$ is called the \textit{half-twist invariant} associated with the loop \nolinebreak $e$. 
\end{definition}
Existence of half-twist invariants is related to the presence of so-called half twists in the automorphism group of a pants decomposition (see the next section).

Note that one has to consider the invariant $\mathrm{ht_e} $ for each loop $e$ in the reduced graph $\bar \Gamma_F$, so that there are exactly $2^k$ possible values of this invariant, where 
$k$ is the number of loops in $\bar \Gamma_F$ for a fixed colored pants decomposition.

\par
\begin{remark}
{\rm
More formally, the half-twist invariant can be defined as the isomorphism
\begin{align*}
\Hom_1\left({ P(v) \cup C(e)}, \Z\right) \, / \,\, \Z\,[C(e)]  \to  \Hom_1\left( e \cup v, \Z\right).
\end{align*}
induced by the projection $\pi$. Since both groups are isomorphic to $\Z$, there are two such isomorphisms, and the half-twist invariant may take two values.
}
\end{remark}

As we show below, a complete list of invariants of the $\SDiff_0(M)$ action on simple Morse function consists of a measured Reeb graph, colored pants decomposition, and half-twist invariants for each of the loops in the reduced Reeb graph.

\medskip

\subsection{Action of the mapping class group on pants decompositions}
Let $M$ be a closed connected surface of genus $\varkappa \geq 2$, and let $\mathcal P(M)$ be the set of all possible pants decompositions of $M$, considered up to isotopy. Then there is a natural action of the mapping class group $\MCG(M)$ on the set $\mathcal P(M)$. 
The following description of  orbits of this action can be found, e.g., in \cite{Putman, Wolf}.

\begin{theorem}\label{wolfThm}
Two pants decompositions $\pazocal P_1$ and $\pazocal P_2$ belong to the same orbit of the $\MCG(M)$ action if and only if the associated graphs $\Gamma(\pazocal P_1)$ and $\Gamma(\pazocal P_2)$ are isomorphic. Moreover, any isomorphism $\phi \colon \Gamma(\pazocal P_1) \to \Gamma(\pazocal P_2)$ gives rise to a certain mapping class $\Phi \in \MCG(M)$.
\end{theorem}

Let $\pazocal P$ be a pants decomposition of $M$ given by 
non-oriented cycles  $\{ C_1, \dots, C_{3\varkappa-3}\}$. Following Wolf \cite{Wolf}, we define 
the pointwise stabilizer of $\pazocal P$ as the set of mapping classes which map 
every cycle $C_i$ to itself:
$$
\Stab_{pw}(\pazocal P) := \{ \Phi \in \MCG(M) ~|~\Phi(C_i) = C_i \mbox{ for all } C_i \in \pazocal P\}.
$$
  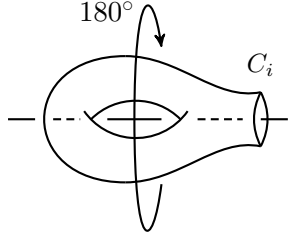
\begin{figure}[t]
\centerline{
\begin{tikzpicture}[thick, scale=1.2]
\draw (2,2.7)  .. controls (2.6, 2.7) and (3, 2.2)  ..  (3.5,2.3);
\draw (2,1.3)  .. controls (2.6, 1.3) and (3, 1.8)  ..  (3.5,1.7);
\draw (2,2.7)  .. controls (0.8,2.7) and (0.8, 1.3)  ..  (2,1.3);
\draw   (2.61,2) arc (45:135:0.7cm);
\draw   (2.69,2.09) arc (-35:-145:0.7cm);
\draw   (3.5,2.3) .. controls (3.6, 2.1) and  (3.6, 1.9) ..  (3.5,1.7);
\draw   (3.5,2.3) .. controls (3.4, 2.1) and  (3.4, 1.9) ..  (3.5,1.7);
\draw[->](2.1,2)  .. controls (2.1, 3.5) and (2.3, 3.5) ..  (2.4,2.8);
\draw(2.1,2)  .. controls (2.1, 0.5) and (2.3, 0.5) ..  (2.4,1.28);
\draw (1.8,2) -- (2.4,2); 
\draw [densely dashed]  (2.8,2) -- (3.3,2); 
\draw  (3.5,2) -- (3.8,2); 
\draw (0.7,2) -- (1,2); 
\draw [densely dashed]  (1.2,2) -- (1.5,2); 
\node at (1.8, 3.2) (nodeA) {$180^{\circ}$};
\node at (3.5,2.6)(nodeB) { \small $C_i$};
\end{tikzpicture}
}
\caption{Half twist about a genus-$1$-separating curve.}\label{halftwist}
\end{figure}
A cycle $C_i$ is called \textit{genus-$1$-separating} if $M \setminus C_i = M_1 \sqcup M_2$ and either $M_1$ or $M_2$ has genus one, i.e. it is a torus with a hole, see Figure \ref{loop}.
For each genus-$1$-separating cycle $C_i$, there is an associated half twist, that is a mapping class which twists a genus-one component of the complement $M \setminus C_i$ by $180$ degrees and is isotopic to the identity on the second component, see Figure \ref{halftwist} ({also remark that the square of a half twist is a Dehn twist).}
Note that such separating cycles are in one-to-one correspondence with loops in 
$\Gamma(\pazocal P)$, provided that not both $M_1$ and $M_2$ are of genus one. 
We also note that if both $M_1$ and $M_2$ have genus one (which is only possible if $M$ has genus two and the pants decomposition is the one depicted in Figure \ref{pretzel0}), then there are two distinct half twists about the curve $C_i$: one twisting $M_1$, and the other twisting $M_2$. 
Thus for an arbitrary surface $M$ of genus $\ge 2$ and its pants decomposition without exception there
is a one-to-one correspondence between half twists and loops in $\Gamma(\pazocal P)$.

Now without loss of generality, assume that the curves $C_1, \dots, C_k$ are genus-$1$-separating, while the curves $C_{k+1}, \dots, C_{3\varkappa - 3}$ are not. We refer to
\cite{Wolf} for the following result on the structure of the stabilizer subgroup.

\begin{lemma}\label{Wolf2}
Assume that $\varkappa \geq 3$. Then the pointwise stabilizer $\Stab_{pw}(\pazocal P)$ is generated by half twists about the curves $C_1, \dots, C_k$
and Dehn twists about the curves $C_{k+1}, \dots, C_{3\varkappa - 3}$.
\end{lemma}
  \begin{figure}[t]
\centerline{
\begin{tikzpicture}[thick, scale=0.8]
\draw   (1.68,3.2) arc (60:300:1.46cm);
\draw   (2.32,3.2) arc (120:-120:1.46cm);
\draw (1.68,3.2)  .. controls (1.9, 3.1) and (2.1, 3.1) ..  (2.32,3.2) ;
\draw (1.68,0.67)  .. controls (1.9, 0.77) and (2.1, 0.77) ..  (2.32,0.67) ;
    \draw   (1,1.3) arc (260:100:0.6cm);
    \draw   (0.8,1.4) arc (-80:80:0.5cm);
        \draw   (3.25,1.3) arc (260:100:0.6cm);
    \draw   (3.05,1.4) arc (-80:80:0.5cm);
\draw    [densely dotted] (2.75, 1.95) arc (65:115:1.85cm);
\draw   (2.75, 1.95) arc (-65:-115:1.85cm);
\draw   [densely dotted] (4.5, 1.95) arc (60:120:1cm);
\draw (4.5, 1.95) arc (-60:-120:1cm);
\draw   [densely dotted]  (0.5, 1.95) arc (60:120:1cm);
\draw  (0.5, 1.95) arc (-60:-120:1cm);
\draw[->](2.3,2)  .. controls (2.3, 4.5) and (2.1, 4.5) ..  (2.1,3.2);
\draw(2.3,2)  .. controls (2.3, -0.5) and (2.1, -0.5) ..  (2.1,0.7);
\node at (2.75, 3.9) (nodeF) {$180^{\circ}$};
\node at (0.1, 1.5) (nodeG) {$C_1$};
\node at (2, 1.5) (nodeH) {$C_2$};
\node at (4, 1.5) (nodeI) {$C_3$};
\node at (0.8, 3.7) (nodeK) {$P_1$};
\node at (0.8, 0.1) (nodeL) {$P_2$};
\end{tikzpicture}
}
\caption{Hyperelliptic involution.}\label{pretzel1}
\end{figure}
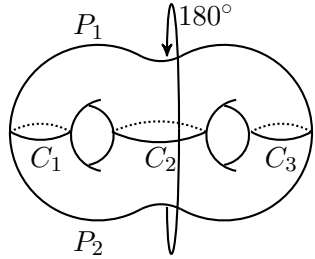

As it is easy to see from the proof, the lemma is true for $\varkappa = 2$ as well, unless $\pazocal P$ 
is the pants decomposition depicted in Figure \ref{pretzel}. 
This is the only case when each of the cycles is mapped to itself but the map on vertices can be nontrivial. For this pants decomposition, the pointwise stabilizer has one more generator which is depicted in Figure \ref{pretzel1}; this mapping class is known as the \textit{hyperelliptic involution}. Note that the hyperelliptic involution indeed preserves the isotopy classes of the curves $C_1, C_2, C_3$, but interchanges the pairs of pants $P_1, P_2$. This leads us to the following definition:
\begin{align*}
\Stab_{0}(\pazocal P) := \{ \Phi \in \MCG(M) ~|~\Phi(C_i) = C_i &\mbox{ for all } C_i \in \pazocal P; \Phi(P_j) = P_j\mbox{ for all } P_j \in \pazocal P\}\,,
\end{align*}
where  $C_1, \dots, C_{3\varkappa-3}$ are the cycles defining $\pazocal P$, and $P_1, \dots, P_{2\varkappa - 2}$ are the pairs of pants of $\pazocal P$. 
This consideration implies the following proposition.

\begin{proposition}\label{stab}
Let $M$ be a closed connected surface of genus $\varkappa \geq 2$. Assume that  $\pazocal P$ is a pants decomposition of $M$. Further, assume that the curves $C_1, \dots, C_k \in \pazocal P$ are genus-$1$-separating, and the curves $C_{k+1}, \dots, C_{3\varkappa - 3}  \in \pazocal P$ are not. Then $\Stab_0(\pazocal P)$ is generated by half twists about the curves $C_1, \dots, C_k$
and Dehn twists about the curves $C_{k+1}$, $\dots$, $C_{3\varkappa - 3}$.
\end{proposition}

Note that the group $\Stab_{0}(\pazocal P)$ is not Abelian but almost Abelian (i.e. it has an Abelian subgroup of finite index). Namely, the group $\Stab_{0}(\pazocal P)$ has an Abelian normal subgroup $K(\pazocal P) \simeq \Z^{3\varkappa - 3}$ generated by Dehn twists about  the curves $C_1, \dots, C_{3\varkappa - 3}$, and 
$$
\Stab_{0}(\pazocal P) / K(\pazocal P) \simeq \Z_2^k\,,
$$
where $k$ is the number of loops in the graph $\Gamma(\pazocal P)$.
Also note that if the pants decomposition $\pazocal P$ is obtained from a simple Morse 
function $F$, then the subgroup $K(\pazocal P) \subset \Stab_{0}(\pazocal P)$ acts trivially 
on half-twist invariants defined in the previous section, while each generator of the quotient 
group $\Z_2^k$ changes the value of the corresponding invariant.


\medskip

\subsection{Classification of simple Morse functions under the $\SDiff_0(M)$ action: higher genus case}\label{sect:high-genus}

Now we are ready to give the definition of a graph frozenness for higher genus.
\begin{definition} \label{def:freezing-any}
Let  $M$ be a closed connected symplectic surface of genus $\varkappa \geq 2$.
A measured Reeb graph $\Gamma$ compatible with $M$ is \textit{frozen} into $M$ if it is endowed with the following additional data:
\begin{longenum}
\item a pants decomposition $\pazocal P$ of $M$ colored 
by the reduced Reeb graph $\bar \Gamma$;
\item half-twist invariant $\halfTwist_e$ for each loop $e \subset \bar \Gamma$.
\end{longenum}
\end{definition}

\begin{definition}
Let  $(\Gamma_1,\pazocal P_1,  \{ \halfTwist_{e,1}\})$ and 
$(\Gamma_2,\pazocal P_2,  \{ \halfTwist_{e,2}\})$  be two measured Reeb graphs 
frozen into the same surface $M$.
We say that the frozen Reeb graphs $\Gamma_1$ and $\Gamma_2$ are isomorphic if there exists an isomorphism $\phi \colon \Gamma_1 \to \Gamma_2$ of measured Reeb graphs such that
 \begin{longenum}
 \item
$\phi$ gives rise to an isomorphism of the corresponding colored pants decompositions
$\pazocal P_1$ to $\pazocal P_2$;
 \item 
 $\phi$ intertwines half-twist invariants  $\{ \halfTwist_{e,1}\}$ and $\{ \halfTwist_{e,2}\}$.
 \end{longenum}
\end{definition}

%

\begin{remark}
{\rm
Definition \ref{def:freeze-torus} of Reeb graphs frozen into torus can be regarded as a particular case 
of Definition \ref{def:freezing-any}. Indeed, if $M=\T^2$ the reduced graph $\bar \Gamma_F$
 is a circle for any simple Morse function $F$. Then, although one does not have a pants 
 decomposition of the torus, one needs to fix the image of the cycle corresponding to the edge $e$, 
 which boils down to fixing a surjective homomorphism of the homology groups 
 $\pi_* \colon \Hom_1(\T^2,\Z) \to \Hom_1(\bar\Gamma_F, \Z)$. }
\end{remark}

\begin{theorem}\label{Sdiff0action}
Under the above definition of freezing, Theorem \ref{Sdiff0Torus} holds true for arbitrary surfaces of genus $\varkappa \geq 1$: for a symplectic surface $M$ of any genus  there is a one-to-one correspondence between simple Morse functions on $M$, considered up to symplectomorphism isotopic to the identity, and (isomorphism classes of) measured Reeb graphs frozen into $M$. 
\end{theorem}

\begin{proof}
Let us prove the first statement of Theorem \ref{Sdiff0Torus} for any genus $\varkappa \geq 2$. The implication (a) $\Rightarrow$ (b) is obvious, so it suffices to prove the implication (b) $\Rightarrow$ (a). Assume that $\phi \colon \Gamma_F \to \Gamma_G $ is an isomorphism of measured Reeb graphs frozen into $M$. By Theorem \ref{thm1}, it can be lifted to a symplectomorphism $
 \Phi' \colon M \to M$ such that $\Phi'_* F = G$. Since the isomorphism of reduced Reeb graphs induced by $\phi$ identifies colored pants decompositions $\pazocal P_F$ and $\pazocal P_G$, we have $\Phi' \in \Stab_0(\pazocal P_F) = \Stab_0(\pazocal P_G)$. Therefore, by Proposition \ref{stab} the mapping class of $\Phi'$ is a finite composition of half twists and Dehn twists about connected components of $F$-levels. 
 
 Furthermore, since  $\phi$ intertwines half-twist invariants  $\{ \halfTwist_{e,1}\}$ and 
 $\{ \halfTwist_{e,2}\}$,  for every loop $e$ in the reduced Reeb graph 
 $\bar \Gamma_1$ this means commutativity of the following diagram:
 \begin{equation}\label{cdHighGen}
\centering
\begin{tikzcd}[column sep=-20pt, row sep = large]
&  \mbox{orientations of } C_1(e) = C_2(\phi(e)) &\\
 \mbox{orientations of } e \arrow{rr}{\phi}\arrow{ur}{\halfTwist_{e,1}} & &  \mbox{orientations of } \phi(e) \arrow{ul}[swap]{\halfTwist_{\phi(e), 2}}.
\end{tikzcd}
\end{equation}
This implies that the mapping class of $\Phi'$ actually lies in the normal subgroup of 
$ \Stab_0(\pazocal P_F)$ generated by Dehn twists. 
Finally, we can get rid of Dehn twists in the same way as in the proof of Theorem \ref{Sdiff0Torus}.
\par
 Now, let us prove the second statement. By Theorem \ref{thm1}, there exists a simple Morse 
 function $F' \colon M \to \R$ such that the measured Reeb graph associated 
 with $F'$ is $\Gamma$. Of course, the pants decomposition associated with $F'$ does not 
 have to coincide with the one prescribed by freezing. However, they have the same graphs, 
 so by Theorem \ref{wolfThm} there exists a mapping class which maps one of these pants decompositions into the other one. Taking a symplectic diffeomorphism $\Phi$ belonging to this mapping class, we obtain a function $F = \Phi_*F'$ such that the pants decomposition associated with $F$ is as desired. Further, by composing $\Phi$ with a suitable number of half twists, we adjust the values of half-twist invariants. As a result, we obtain a function $F$ with desired properties.
\end{proof}

\begin{corollary}
A complete set of invariants of a simple Morse function on a closed symplectic surface $M$
with respect to the $\SDiff_0(M)$-action consists of invariants of a measured Reeb graph of the function, 
a choice of a colored pants decomposition of $M$, and a $\Z^k_2$-valued invariant of possible orientations of the cycles described above.
\end{corollary}

\bigskip


\section{Classification of coadjoint orbits of symplectomorphism groups}

\subsection{Graph's anti-derivatives, or circulation functions}\label{sect:circ}
Recall that the regular dual $\SVect^*(M)$ of the Lie algebra $\SVect(M)$ of divergence-free
vector fields on a manifold $M$  is identified with the space $\Omega^1(M) / \diff \Omega^0(M)$ 
of smooth $1$-forms modulo exact $1$-forms on $M$. 
The coadjoint action of a $\SDiff(M)$ on $\SVect^*(M)$ is given by the change of coordinates in (cosets of) 1-forms on $M$ by means of a volume-preserving diffeomorphism:
$$
\Ad^*_\Phi \,[\oneform] = [\Phi^*\oneform].
$$
In what follows, the notation $[\alpha]$ stands for the coset of 1-forms 
$\alpha$ in $\Omega^1(M) / \diff \Omega^0(M)$. In particular, if the form $\alpha$ is closed, then $[\alpha]$ is the cohomology class of $\alpha$.
\par
For a symplectic surface $(M,\omega)$ consider the surjective mapping
$$
\Diff \colon \Omega^1(M) / \diff \Omega^0(M) \to 
C^\infty_0(M):=\left\{ F \in C^\infty(M) ~|~\int_M  \!\!F\omega = 0\right\}
$$
given by taking the vorticity function, 
$$
\Diff[\oneform] = \frac{\diff \oneform}{\omega}.
$$
(One can view this map as taking the vorticity function $\hat\xi= \diff \oneform/\omega$
of a vector field $v=\alpha^\sharp$, as we discussed 
in Section \ref{sect:euler}.)
Clearly, if cosets $[\oneform], [\oneformtwo] \in \SVect^*(M)$ belong to the same coadjoint orbit, then the functions $\Diff[\oneform]$ and $  \Diff[\oneformtwo]$ are conjugated by a symplectic diffeomorphism. In particular, if  $\Diff[\oneform]$ is a simple Morse function, then  so is $  \Diff[\oneformtwo]$.

\begin{definition}
{\rm
We say that a coset of 1-forms $[\oneform] \in \SVect^*(M)$ is \textit{generic} if $\Diff[\oneform]$ is a simple Morse function.
A coadjoint orbit $\orbit \subset \SVect^*(M)$ is \textit{generic} if any coset $[\oneform] \in \orbit$ is generic (equivalently, if at least one coset $[\oneform] \in \orbit$ is generic).
}
\end{definition}

\begin{remark}
{\rm
Assume that $[\oneform]$ and $[\oneformtwo]$ belong to the same generic coadjoint orbit. Then the functions $\Diff[\oneform],$ and $  \Diff[\oneformtwo]$ are simple Morse functions which have isomorphic measured Reeb graphs. Therefore, the measured Reeb graph of $\Diff[\oneform]$ is an invariant of the coadjoint action of $\SDiff(M)$ on $\SVect^*(M)$. However, this invariant is not complete. Indeed, assume that  $\Diff[\oneform]$ and $  \Diff[\oneformtwo]$ have isomorphic measured Reeb graphs. Then there exists a symplectic diffeomorphism $\Phi$ such that $\Phi^*\Diff[\oneformtwo] = \Diff[\oneform]$, and thus the $1$-form
$$
\oneformthree = \Phi^*\oneformtwo - \oneform
$$
is closed. However, it is not necessarily exact, so $\oneform$ and $\oneformtwo$ do not 
necessarily belong to the same coadjoint orbit. Nevertheless, we can conclude that the 
moduli space of $\SDiff(M)$ coadjoint orbits corresponding to the same measured 
Reeb graph is finite-dimensional and its dimension is at most $\dim \Hom^1(M,\R)  = 2\varkappa$, 
where $\varkappa$ is the genus of $M$. 
As we show below, this dimension is actually equal to $\varkappa$. 
The reason for a half-dimensional reduction  
is that the symplectic diffeomorphism $\Phi$ that maps $\Diff[\oneform]$ to 
$  \Diff[\oneformtwo]$ is not unique, and we may use this freedom to vary the cohomology 
class of $\oneformthree$ within a $\varkappa$-dimensional subspace of $\Hom^1(M,\R) $.
}
\end{remark}
\par
Let $[\oneform] \in \SVect^*(M)$ be generic, and let $F = \Diff[\oneform]$.  
Consider the measured Reeb graph $\Gamma_F$. Let $\pi \colon M \to \Gamma_F$ 
be the natural projection. Take any point $x$ lying in the interior of some edge 
$e \in \Gamma_F$. Then $\pi^{-1}(x)$ is a circle $C$. It is naturally oriented 
as the boundary of the set of smaller values. The integral of $\oneform$ 
over $C$ does not depend on the choice of a representative $\oneform \in [\oneform]$.
Thus, we obtain a function
$$
\circulation \colon \Gamma_F \setminus V( \Gamma_F) \to \R
$$
given by 
$$
\circulation(x) = \oint_{\mathclap{\pi^{-1}(x)}} \,\oneform\,,
$$
where  $V( \Gamma_F)$ is the set of vertices  of the graph $ \Gamma_F$.
Note that in the presence of a metric on $M$, the value $\circulation(x)$ is the circulation over the level 
$\pi^{-1}(x)$ of the vector field  $\alpha^\sharp$ dual to the 1-form $ \oneform$\,.

\begin{proposition}\label{circProperties}
The function $\circulation$ has the following properties.
\begin{longenum}
\item Assume that $x,y$ are two interior points of some edge $e \subset \Gamma_F$. Then
\begin{align}\label{stokes}
\circulation(y) - \circulation(x) = \int\limits_x^y f\diff \mu.
\end{align}
\item Let $v$ be a $1$-valent vertex of $\Gamma_F$. Then
\begin{align}\label{1valentcirc}
\lim_{x \to v} \circulation(x) = 0.
\end{align}
\item Let $v$ be a $3$-valent vertex of $\Gamma_F$. Let $e_0$ be the trunk of $v$, and let $e_1, e_2$ be the branches of $v$. Let also $x_i \in e_i$. Then
\begin{align}\label{3valentcirc}
\lim_{\mathclap{x_0 \to v}}\, \circulation(x_0) = \lim_{\mathclap{x_1 \to v}} \,\circulation(x_1)  + \lim_{\mathclap{x_2 \to v}} \,\circulation(x_2). 
\end{align}
\end{longenum}
\end{proposition}
\begin{proof}
The proof is straightforward and follows from the Stokes formula and additivity of the circulation integral.
\end{proof}

\begin{definition}
{\rm
Let  $(\Gamma, f, \mu)$ be a measured Reeb graph. Any function $\circulation \colon \Gamma \setminus V(\Gamma) \to \R$ satisfying properties listed in Proposition \ref{circProperties} is called a \textit{circulation function}  (or \textit{an
anti-derivative}). A measured Reeb graph endowed with a circulation function is called a \textit{circulation graph}  $(\Gamma, f, \mu, \circulation)$.

Note that the function $f$ on the graph can be recovered from the circulation function 
$\circulation$, as formula (\ref{stokes}) implies: $f= \diff \circulation/\diff \mu$.
Two circulation graphs are isomorphic if they are isomorphic as measured Reeb graphs, and the isomorphism between them preserves the circulation function. 

Above we associated a circulation graph   $\Gamma_{[\oneform]} :=(\Gamma, f, \mu, \circulation)$ to 
any generic coset $[\oneform] \in \SVect^*(M)$.
\par
Similarly, a frozen measured Reeb graph endowed with a circulation function is called a \textit{frozen circulation graph}. Two frozen circulation graphs are isomorphic if they are isomorphic as  measured Reeb graphs frozen into a surface, and the isomorphism between them preserves the circulation function.
}
\end{definition}

\begin{proposition}\label{circFuncs}
Let $(\Gamma, f, \mu)$  be a measured Reeb graph. 
\begin{longenum}
\item
The graph $\Gamma$ admits a circulation function if and only if
\begin{align}\label{zeroMass}
 \int_{{\Gamma}} \!f(x) \diff \mu = 0.
\end{align}
\item
If $\Gamma$ admits a circulation function, then the set of circulation functions on $\Gamma$ is an affine space of dimension equal to the first Betti number of $\Gamma$.
\end{longenum}
\end{proposition}
\begin{proof}
Let us prove the first statement. Assume that $\Gamma$ admits a circulation function. Let $e  \in \Gamma$ be an edge of $\Gamma$ going from $v$ to $w$, and let $x \in e$. Let $\circulation^-(e)$ and $\circulation^+(e)$ be the limits of $\circulation(x)$ as $x$ tends to $v$ and $w$, respectively. We have
\begin{align*}
 \int_\Gamma \! f(x) \diff \mu =
\sum_{\mathclap{e \,\in \, E(\Gamma)}}\, \left(\circulation^+(e) - \circulation^-(e) \right).
\end{align*}
On the other hand, properties \eqref{1valentcirc} and  \eqref{3valentcirc} imply that the sum at the right hand side of the latter equation vanishes, and hence \eqref{zeroMass} holds.
\par
Conversely, assume that  \eqref{zeroMass} holds. 
By Theorem \ref{thm1}, one can construct a symplectic surface $M$ and a simple Morse function $F \colon M \to \R$ such that the measured Reeb graph of $F$ is $(\Gamma, f, \mu)$. Since $F$ has zero mean, we have $F =\Diff[\alpha]$ for some $1$-form $\alpha$ on $M$. Integrating $\alpha$ over connected components of level sets of $F$, we obtain a circulation function on $\Gamma$, as desired.
\par
Now, let us prove the second statement. Let $\circulation$ and $\circulationtwo$ be two circulation functions. Then, in view of property \eqref{stokes}, their difference is constant on each edge. Consider the $1$-chain
\begin{align}\label{1chain}
\circulation - \circulationtwo=\sum_{\mathclap{e \,\in\, E(\Gamma)}}\,\,(\circulation(x_e) - \circulationtwo(x_e))e\,,
\end{align}
where $E(\Gamma)$ is the set of edges of $\Gamma$, and $x_e \in e$ is any interior point. Properties \eqref{1valentcirc} and~\eqref{3valentcirc} imply that $\circulation - \circulationtwo$ is a $1$-cycle.
On the other hand, if we add a $1$-cycle to a circulation function, we obtain a circulation function. Therefore, the space of circulation functions is an affine space with underlying vector space $ \Hom_1(\Gamma, \R)$, q.e.d.
\end{proof}


\medskip

\subsection{Coadjoint orbits of $\SDiff(M)$ and $\SDiff_0(M)$}\label{sect:coadj-sdiff}

\begin{theorem}\label{thm4} \label{thm:sdiffM}
Let $M$ be a closed connected symplectic  surface. Then generic coadjoint orbits of $\SDiff(M)$ are in one-to-one correspondence with (isomorphism classes of) circulation graphs  
$(\Gamma, f, \mu, \circulation)$  compatible with $M$.
In other words, the following statements hold:
\begin{longenum}
\item For a symplectic surface $M$ and generic cosets $[\oneform], [\oneformtwo] \in \SVect^*(M)$ 
the following conditions are equivalent:
\begin{longenum} \item $[\oneform]$ and $ [\oneformtwo]$ lie in the same orbit of the $\SDiff(M)$ coadjoint action;  \item circulation graphs $\Gamma_{[\oneform]}$ and $\Gamma_{[\oneformtwo]}$ corresponding to the cosets $[\oneform]$ and $[\oneformtwo]$ 
are isomorphic.\end{longenum}
\item For each circulation graph $\Gamma$ which is compatible\footnote{See Definition \ref{def:compatible} for compatibility of a graph and a surface.}
 with $M$, there exists a generic $[\oneform] \in   \SVect^*(M)$ such that  $\Gamma_{[\oneform]} =(\Gamma, f, \mu, \circulation)$.
\end{longenum}
\end{theorem}

Similarly, we have the following result:

\begin{theorem}\label{thm5} \label{thm:sdiff0M}
Let $M$ be a closed connected symplectic surface. Then generic coadjoint orbits of $\SDiff_0(M)$ are in one-to-one correspondence with (isomorphism classes of) circulation graphs frozen into $M$.
\end{theorem}

The proofs of these two theorems are  identical, with the only difference 
that the proof of Theorem \ref{thm4} is based on Theorem \ref{thm1}, 
while the proof of Theorem \ref{thm5} is based on Theorem~\ref{Sdiff0action}. 
For this reason, we shall only prove Theorem \ref{thm4}.
We start with the following preliminary lemma.

\begin{lemma}\label{graphCohomology}
Let $M$ a closed connected oriented surface, and let $F$ be a simple Morse function on $M$. 
Assume that $[\gamma] \in  \Hom^1(M,\R)$ is such that the integral of $\gamma$ over 
any connected component of any $F$-level vanishes. Then there exists a $C^\infty$ 
function $H \colon M \to \R$ such that the $1$-form $H\diff F$ is closed, and its cohomology class 
is equal to $[\gamma]$. Moreover, $H$ can be chosen in such a way that the ratio
$H/F$ is a smooth function.
\end{lemma}

\begin{proof}
Since the integral of $[\gamma]$ over any connected component of any $F$-level vanishes, 
the cohomology class $[\gamma]$ on $M$ belongs to the image of the inclusion
$$
i \colon \Hom^1(\Gamma_F, \R) \to \Hom^1(M, \R).
$$
Let $\alpha$ be a $1$-cochain on the graph $\Gamma_F$ representing the cohomology class $i^{-1}([\gamma])$. Then $\alpha$ is a real-valued function on the set of edges of $\Gamma_F$. 
Recall that the function $f$ is the pushforward  of the function $F$ to the graph $\Gamma_F$.
Consider a continuous function $ h \colon \Gamma_F \to \R$ such that 
\begin{longenum}
\item it is a smooth function of $f$ in a neighborhood of each point $x \in \Gamma_F$;
\item it vanishes if $f$ is sufficiently close to zero;
\item for each edge $e$, we have
$$
\alpha (e) = \int_e \! h\diff f.
$$
\end{longenum}
Obviously, such  a function does exist. Now, lifting $ h$ to $M$, we obtain a smooth function $H$ with 
the desired properties.
\end{proof}

\begin{proof}[Proof of Theorem \ref{thm4}]
Let us prove the first statement. The implication (a) $\Rightarrow$ (b) is immediate, 
so it suffices to prove the implication (b) $\Rightarrow$ (a). 
Let $\phi \colon \Gamma_{[\oneform]} \to \Gamma_{[\oneformtwo]}$ be an isomorphism 
of circulation graphs. By Theorem \ref{thm1}, $\phi$ can be lifted to a symplectomorphism 
$ \Phi \colon M \to M$ that  maps the function $F = \Diff[\oneform]$ 
to the function $G = \Diff[\oneformtwo]$. Therefore, the $1$-form $\oneformthree$ defined by
$$
\oneformthree =  \Phi^*\oneformtwo - \oneform 
$$
is closed. 
\par
Assume that $ \Psi \colon M \to M $ is a symplectomorphism which maps the function $F$ to itself and is isotopic to the identity. Then the composition $\widetilde\Phi = \Phi\circ\Psi^{-1}$ maps $F$ to $G$, and
$$
[\widetilde\Phi^*\oneformtwo - \oneform] = [\Phi^*\oneformtwo - \Psi^*\oneform] = [\gamma] - [\Psi^*\oneform - \oneform].
$$
We claim that $\Psi$ can be chosen in such a way that $\widetilde\Phi^*\oneformtwo - \oneform$ is exact, i.e.  one has the equality of the cohomology classes
$$
 [\Psi^*\oneform - \oneform] = [\gamma].
$$
Moreover, let us show that there exists a time-independent symplectic vector field  $X$ that
 preserves $F$ and satisfies
\begin{align}\label{homotopy2}
 [\Psi_t^*\oneform - \oneform] = t[\gamma]\,,
\end{align}
where $\Psi_t$ is the phase flow of $X$. 
Differentiating \eqref{homotopy2} with respect to $t$, we get in the left-hand side
$$
 [\Psi_t^*L_X\oneform] =  [L_X\oneform] =  [i_X\diff \oneform] = [F\cdot i_X \omega]\,,
$$
since $L_X \oneform$ is closed and $\Psi_t^*$ does not change its cohomology class. Thus
\begin{align}\label{homology2}
[F\cdot i_X \omega] = [\gamma].
\end{align}
Since $\Phi$ preserves the circulation function, the integrals of $\oneformthree$ over all connected components of $F$-levels vanish. Therefore, by Lemma \ref{graphCohomology}, there exists a smooth function $H$ such that
$$
[\gamma] = [H\diff F].
$$
Now we set 
$$
X :=  \frac{H}{F} \,\omega^{-1}\diff F.
$$
It is easy to see that the vector field $X$ is symplectic, preserves the levels of $F$, and satisfies the equation \nolinebreak \eqref{homology2}. Therefore, its phase flow satisfies the equation \eqref{homotopy2}, and then the symplectomorphism $\widetilde\Phi = \Phi \circ \Psi_1^{-1}$ for the time-one
map $\Psi=\Psi_1$  has the required properties.
\par
Now, let us prove the second statement. By Theorem \ref{thm1}, there exists a simple Morse function $F \colon M \to \R$ such that the measured Reeb graph of $F$ is $(\Gamma, f, \mu)$. 
Since the graph $\Gamma$ admits a circulation function, Proposition \ref{circFuncs} implies that $F$ has zero mean. Therefore, there exists a $1$-form $\alpha \in \Omega^1(M)$ such that $\Diff[\alpha] = F$. 
Further, if $\gamma$ is a closed $1$-form, then $\Diff[\alpha + \gamma] = F$ as well. 
For any $1$-form $\tilde\alpha$ such that  $\Diff[\tilde\alpha] = F$, let $\circulation_{\tilde\alpha}$ denote the corresponding circulation function on $\Gamma$. Consider the mapping
$$
\rho \colon \Hom^1(M,\R) \to \Hom_1(\Gamma,\R) 
$$
given by
$$
\rho\left([\gamma]\right) = \circulation_{\alpha + \gamma} - \circulation_{\alpha}\,,
$$
where the right hand side is defined by equation \eqref{1chain}. The mapping $\rho$ can be written as
$$
\rho([\gamma]) =\!\! \sum_{e \,\in\, E(\Gamma)} \left(\int_{{C(e)}}\!\!\gamma\right) e\,,
$$
where $C(e) = \pi^{-1}(x_e)$ and $x_e \in e$ is any interior point of the edge $e$. 
Therefore, the kernel of the homomorphism $\rho$ consists of those cohomology classes 
which vanish on cycles homologous to connected components of regular $F$-levels, 
and $\dim \Ker \rho = \varkappa$, where $\varkappa$ is the genus of $M$. 
So, by the dimension argument, the homomorphism $\rho$ is surjective. 
(Also note that the mapping $\rho$ can be written as $\pi_*\circ p$, where 
$p \colon  \Hom^1(M,\R) \to  \Hom_1(M,\R)$ is the Poincar\'{e} duality, and 
$\pi_* \colon  \Hom_1(M,\R) \to  \Hom_1(\Gamma,\R)$ is the epimorphism 
induced by the projection $\pi$.)
\par
Now, since the homomorphism $\rho$ is surjective, one can find a closed 1-form $\gamma$ such that 
$$
\rho([\gamma]) = \circulation - \circulation_{\alpha},
$$
where $\circulation$ is a given circulation function on $\Gamma$, and therefore 
$
 \circulation_{\alpha + \gamma} = \circulation,
$
as desired.
\end{proof}

\bigskip


\section{Coadjoint orbits of the group of Hamiltonian diffeomorphisms}\label{HamSection}

Let $M$ be a symplectic manifold, and let $\Phi \in \SDiff_0(M)$ be a symplectic diffeomorphism 
of $M$ isotopic to the identity. Recall that $\Phi$ is called \textit{Hamiltonian} if a path $\Phi_t$ 
joining the identity $\Phi_0=id$ and $\Phi_1=\Phi$ can be chosen in such a way that 
the vector field
$$
X_t := \left(\diffXp{t}\Phi_t\right)\circ \Phi_t^{-1}
$$
is Hamiltonian for every $t$. In other words, a symplectic diffeomorphism is Hamiltonian if it is a time-one map for a suitable time-dependent Hamiltonian vector field. Smooth Hamiltonian diffeomorphisms of $M$ form an infinite-dimensional group, which we denote by ${\rm Ham}(M)$. The aim of this section is to describe generic coadjoint orbits of ${\rm Ham}(M)$ in the case of a two-dimensional surface $M$.
\par
The Lie algebra ${\ham}(M)$ of the group ${\rm Ham}(M)$ consists of all smooth Hamiltonian vector fields on $M$. This Lie algebra is naturally isomorphic to the Lie algebra 
$$
\ham(M) \simeq C_0^\infty(M) = \left\{ F \in C^\infty(M) ~\vert~ \int_M  \!\! F\cdot\omega = 0\,\right\}
$$
of all smooth functions with zero mean with respect to the Poisson bracket.
The adjoint action of the group $\Ham(M)$ on its Lie algebra $\ham(M)$ is the natural action of diffeomorphisms on functions. The Lie algebra $\ham(M)$ is endowed with a bi-invariant inner product
$$
( F,G) :=  \int_M  \!\! FG\cdot \omega,
$$
therefore the regular dual  $\ham^*(M)$ is naturally isomorphic to $\ham(M)$, and the group coadjoint orbits coincide with the adjoint ones. Thus, coadjoint orbits of the group $\Ham(M)$ are exactly the orbits of the natural action of this group on functions. 
\par

\begin{remark}
{\rm
Since $\Ham(M)$ is a subgroup of the group $\SDiff_0(M)$ of symplectomorphisms isotopic 
to identity, the $\SDiff_0(M)$ invariants of functions described by Theorems \ref{Sdiff0Torus} 
and \ref{Sdiff0action} are also $\Ham(M)$ invariants. 
Note that the natural projection of the dual spaces $\SVect^*(M)\to \ham^*(M)$, which follows from the
embeddings of the corresponding Lie algebras $\ham(M) \hookrightarrow\SVect(M)$, is nothing but the 
operator $\Diff[\oneform]:=\diff \oneform/\omega$ defined in Section \ref{sect:circ}.
}
\end{remark}

The image of any $\SDiff_0(M)$-orbit contains whole $\Ham(M)$-orbits, as the $\Ham(M)$-orbits are finer than $\SDiff_0(M)$-ones. Additional invariants of the Hamiltonian orbits can be described in terms of  certain flux-type quantities.

Namely, let $F$ and $G$ be two simple Morse functions of $M$  belonging to the same $\SDiff_0(M)$ orbit. Then there exists an isomorphism $\phi \colon \Gamma_F \to \Gamma_G$ of frozen measured Reeb graphs. Let $\varkappa$ be the genus of $M$. Then the first Betti number of $\Gamma_F$ is equal to $\varkappa$. Therefore, one can choose $\varkappa$ edges $e_1, \dots, e_\varkappa$ of $\Gamma_F$ in such a way that $\Gamma_F \setminus \{ e_1, \dots, e_\varkappa \}$ is a maximal sub-tree of $\Gamma_F$. 
In other words, after dropping edges $e_1, \dots, e_\varkappa$ from the graph $\Gamma_F$ 
it  still remains connected but has no cycles. 

 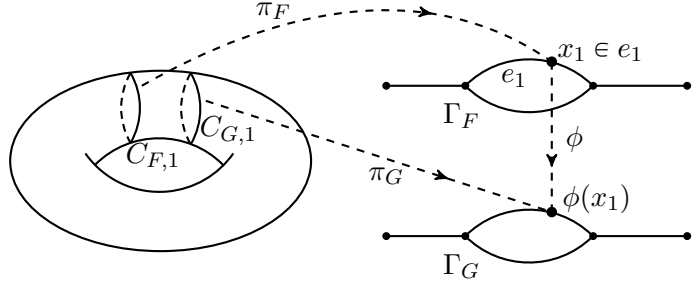
\begin{figure}[t]
\centerline{
\begin{tikzpicture}[thick, rotate = -90]
    \draw (3,-3) ellipse (1.2cm and 2cm);
    \draw   (3.05,-3.87) arc (225:135:1.2cm);
    \draw   (2.9,-4) arc (-55:55:1.2cm);
    \draw  [densely dashed] (1.83,-3.4) arc (-120:-60:0.93cm);
        \draw   (1.83,-3.4) arc (120:60:0.93cm);
            \draw  [densely dashed] (1.83,-2.6) arc (-120:-60:0.95cm);
        \draw   (1.83,-2.6) arc (120:60:0.95cm);
    %
    \node at (2.4,1) () {$\Gamma_F$};
        \node at (4.4,1) () {$\Gamma_G$};
    \node [vertex] at (2,0) (nodeC) {};
    \node [vertex]  at (2,1.05) (nodeD) {};
    \node [vertex] at (2,2.75) (nodeE) {};
    \node [vertex]  at (2,4) (nodeF) {};
        \node [vertex] at (4,0) (nodeC2) {};
    \node [vertex]  at (4,1.05) (nodeD2) {};
    \node [vertex] at (4,2.75) (nodeE2) {};
    \node [vertex]  at (4,4) (nodeF2) {};
    \draw   (nodeC) -- (nodeD);
    \fill (nodeC) circle [radius=1.5pt];
    \fill (nodeD) circle [radius=1.5pt];
    \fill (nodeE) circle [radius=1.5pt];
    \fill (nodeF) circle [radius=1.5pt];
    \draw  (nodeD) arc (-45:44:1.2cm);
    \draw  (nodeD) arc (225:136:1.2cm);
    \draw   (nodeE) -- (nodeF);
        \draw   (nodeC2) -- (nodeD2);
    \fill (nodeC2) circle [radius=1.5pt];
    \fill (nodeD2) circle [radius=1.5pt];
    \fill (nodeE2) circle [radius=1.5pt];
    \fill (nodeF2) circle [radius=1.5pt];
    \draw  (nodeD2) arc (-45:44:1.2cm);
    \draw  (nodeD2) arc (225:136:1.2cm);
    \draw   (nodeE2) -- (nodeF2);
       \node [vertex]  at (1.68,2.2) (nodeG) {};
       \node at (1.55,2.85) () {$x_1 \in e_1$};    
              \node at (1.9,1.7) () {$e_1$};    
               \node at (2.95,-3.1) (nodeI) {$C_{F,1}$};   
                              \node [vertex] at (2,-3.25) (nodeId) {};   
                                                            \node [vertex] at (2.2,-2.4) (nodeJd) {};   
                              \node at (2.6,-2.1) (nodeJ) {$C_{G,1}$};  
                 \draw   [ dashed, ->-] (nodeId) .. controls (1, -1.5) and (0.5,+0.5) .. (nodeG);
          \node at (1, -1.5) () {$\pi_F$};                 
           \fill (nodeG) circle [radius=2pt];
       \node [vertex]  at (3.68,2.2) (nodeG2) {};
                           \draw   [ dashed, ->-] (nodeJd)-- (nodeG2);
                                         \node at (2.7,2.5) () {$\phi$};   
              \node at (3.5,2.8) () {$\phi(x_1)$};     
           \fill (nodeG2) circle [radius=2pt];         
           \draw   [ dashed, ->-] (nodeG) -- (nodeG2);
                     \node at (3.2, 0) () {$\pi_G$};     
\end{tikzpicture}
}
\caption{Construction of the curves $C_{F,i}$ and $C_{G,i}$.}\label{torus2}
\end{figure}

 On each edge $e_i$, we choose an interior point $x_i \in e_i\subset \Gamma_F $ 
 and consider the corresponding images $\phi(x_i)\in \Gamma_G$. Denote by
$$
C_{F,i} = \pi_F^{-1}(x_i), \quad C_{G,i} = \pi_G^{-1}(\phi(x_i))
$$
the corresponding level curves of the functions $F$ and $G$, 
where $\pi_F \colon M \to \Gamma_F$, $\pi_G \colon M \to \Gamma_G$ are canonical projections (see Figure \ref{torus2}). Since the isomorphism $\phi$ preserves freezing, the curves $C_{F,i}$ and $C_{G,i} $  are isotopic for each $i$. 

\begin{definition}
{\rm
The area between the curves $C_{F,i}$ and $C_{G,i}$ is {\it equal to zero} if
\begin{align}\label{areaBetweenCurves}
\int_{\rm Cyl} \!\!\Psi^* \omega = 0
\end{align}
for a smooth mapping $\Psi : {\rm Cyl} \to M$
of the cylinder $\rm{Cyl}$ to $M$ which maps two boundary components of $\rm{Cyl}$ to $C_{F,i}$ and $C_{G,i}$ respectively. 
}
\end{definition}

\begin{lemma}
If the genus of $M$ is $\varkappa \geq 2$, then the integral in the left-hand side does not depend on the choice of the cylinder map $\Psi$. If the genus $\varkappa =1$, i.e., $M=\T^2$, 
the integral is well-defined modulo symplectic area of $\T^2$.
\end{lemma}

\begin{proof}
Indeed, if we have two different maps $\Psi_1, \Psi_2 \colon \rm{Cyl} \to M$ that coincide on boundary components of $C$, then together they can be regarded as a map $ \Psi_{12} \colon \T^2 \to M$ of a torus $\T^2$ to the surface $M$. If the genus of $M$ is bigger than one, then the induced map $ \Psi_{12}^* \colon \Hom^2(M,\R) \to \Hom^2(\T^2, \R)$ is the zero map, therefore we have
$$
\int_{\T^2}  \!\! \Psi_{12}^* \omega = 0,
$$
which implies that
$$
\int_{\rm{Cyl}} \!\! \Psi_1^* \omega = \int_{\rm{Cyl}} \!\! \Psi_2^* \omega.
$$
If the genus of $M$ is one, then we have
\begin{align*}
\int_{\T^2}  \!\!\Psi_{12}^* \omega = n \int_{M}  \!\! \omega
\end{align*}
where $n \in \Z$, and therefore the integral on the left-hand side of \eqref{areaBetweenCurves} is well-defined modulo a multiple of the area of $M$. So, if $M$ has genus one, we define the area between the curves $C_{F,i}$ and $C_{G,i}$ to be equal to zero if 
\begin{align*}
\int_{\mathrm{Cyl}}  \!\!\Psi^* \omega = n \int_{M}  \!\! \omega
\end{align*}
for $n \in \Z$. 
Note that if this identity holds for some choice of $\Psi$, then it is always 
possible to choose another mapping $\Psi$ in such a way that $n = 0$.
\end{proof}

\begin{theorem}\label{HamThm}
Let $F$ and $G$ be two simple Morse functions on a symplectic surface $M$  belonging to the same $\SDiff_0(M)$ 
orbit. Then they belong to the same $\Ham(M)$ coadjoint orbit if and only if there exists an 
isomorphism $\phi \colon \Gamma_F \to \Gamma_G$ between the corresponding  
measured Reeb graphs  frozen into $M$ such that the area between the curves $C_{F,i}$ and $C_{G,i}$  
defined above  is equal to zero  for each $i$.
\end{theorem}

Before  proving Theorem \ref{HamThm} we  recall the notion of the \textit{flux homomorphism}. Let $\Phi \in \SDiff_0(M)$ be a symplectic diffeomorphism isotopic to the identity, and let $\Phi_t$ be smooth a path such that $\Phi_0 = \id, \Phi_1 = \Phi$, and $\Phi_t$ is a symplectic diffeomorphism for every $t \in [0,1]$. Then the vector field
$$
X_t = \left(\diffXp{t}\Phi_t\right)\circ \Phi_t^{-1}
$$
is symplectic
 if the 1-form $\alpha_t = i_{X_t} \omega$
is closed for every $t$. Consider the cohomology class
\begin{align}\label{flux1}
\Flux(\Phi_t) = \int_{0}^1 [\alpha_t] \,\diff t \in \Hom^1(M, \R)\,,
\end{align}
where $[\alpha_t]$ is the cohomology class of $\alpha_t$. 
A priori, $\Flux(\Phi_t)$ may depend on the path $\Phi_t$.

\begin{proposition}
The cohomology class $\Flux(\Phi_t)$  is uniquely determined by the homotopy type of the family
of symplectomorphisms $\Phi_t$.  Explicitly,  
let $a(s)$ be any parametrized loop in $M$, $a(s+1) = a(s)$ 
and let $\rm{Cyl} = \R / \Z \times [0,1]$ be a cylinder understood as a family of such loops.
Then for a map
$
\Psi \colon  {\rm Cyl} \to M
$
given by $\Psi(s,t) = \Phi_t(a(s))$ one has 
\begin{align}\label{fluxViaArea}
\langle\, \Flux(\Phi_t), [a]\,\rangle = \int_{\rm{Cyl}}  \!\!\Psi^* \omega\,,
\end{align}
where $\langle\, , \rangle$ denotes the canonical pairing between $\Hom^1(M, \R)$ and $\Hom_1(M, \R)$. 
\end{proposition}

\begin{proof} Choose a $1$-form $\beta$ on ${\rm Cyl}$ such that $ \xi^* \omega = \diff \beta$, and let $C_t = \{(s,u) \in {\rm Cyl} \mid u= t \}$ be the loop on the cylinder corresponding to parameter $t$. Then 
$$
\int_{\rm Cyl} \!\! \Psi^* \omega = \int_{C_1} \!\! \beta - \int_{C_0} \!\! \beta =  \int_{0}^1 \left( \diffXp{t}\int_{C_t} \beta \right) \diff t.
$$
When computing the integral of $\beta$ over $C_t$, we may assume that $\beta = \Psi^*  \gamma$, where the 1-form $\gamma$ is an antidifferential of $\omega$ defined in the small neighborhood of the curve $\Psi(C_t)$. So,
\begin{align*}
\diffXp{t}\int_{C_t}  \!\!\beta = \diffXp{t}\int_{C_t}  \!\!\Psi^* \gamma &=   \diffXp{t}\int_{\Phi_t(a)} \!\!\gamma = \diffXp{t}\int_{a}\Phi_t^*\gamma =
  \int_{a} \diffXp{t} \Phi_t^*\gamma\\ = 
  \int_{a}  \Phi_t^*L_{X_t}\gamma &=   \int_{a}  \Phi_t^*( i_{X_t}\diff \gamma + \diff i_{X_t}\gamma) = 
 \int_{a}  \Phi_t^*\alpha_t =  \langle \, [\alpha_t], [a] \,\rangle,
\end{align*}
which proves formula \eqref{fluxViaArea}.
\end{proof}
 From \eqref{fluxViaArea}, it follows that $\Flux(\phi_t)$ depends only on the homotopy type of $\Phi_t$, since the integral in the right-hand side does not change under a deformation of the map $\Psi$ fixed on the boundary circles of $\rm{Cyl}$. This way we obtain a map
\begin{align*}
 \widetilde{\Flux}\colon {\widetilde\SDiff(M)} \to \Hom^1(M, \R)
\end{align*}
where ${\widetilde\SDiff(M)}$ is the universal cover of ${\SDiff_0(M)}$. It is easy to see that this map is a homomorphism of groups.  Let $L$ be the image of the fundamental group of $\SDiff_0(M)$ under the map $\widetilde \Flux$.
\begin{definition}
The homomorphism 
\begin{align*}
 {\Flux}\colon {\SDiff_0(M)} \to \Hom^1(M, \R) \,/ \,L\,,
\end{align*}
obtained by descending the mapping $\widetilde \Flux$, is called the \textit{flux homomorphism}. 
\end{definition}
It follows from the construction of  ${\Flux}$, that if $\Phi \in \Ham(M)$, then $  {\Flux}(\Phi) = 0$. The converse result is also true:
 
 \begin{theorem}[Banyaga \cite{banyaga}]\label{BanThm} Let ${\Flux}\colon {\SDiff_0(M)} \to \Hom^1(M, \R) \,/ \,L$ be the flux homomorphism. Then
 $\Ker \Flux = \Ham(M)$.
 \end{theorem}
 
 \begin{remark}
 {\rm
 Let us comment on the structure of the group $L$ when $M$ is a surface. 
 As follows from the Moser theorem, the group  $\SDiff_0(M)$ is a deformation retract 
 of the ambient group $\Diffeo_0(M)$ of all diffeomorphisms of $M$ isotopic to the identity. 
 Furthermore, the group $\Diffeo_0(M)$ is contractible if genus of $M$ is $\varkappa \geq 2$, 
 and it is a deformation retract to the subgroup of linear automorphisms for $\varkappa=1$ 
 (see \cite{earle, gramain, MCG}). 
 Therefore, if the genus of $M$ is at least two, then the group $\SDiff_0(M)$ is contractible, 
 and thus $L= 0$. If $M$ is a torus, then  $\SDiff_0(M)$  is a deformation retract to linear 
 automorphisms, which easily implies that 
\begin{align*}
 L = \left(\int_M  \!\!\omega \right)\Hom^1(M,\Z)\,.\end{align*}
Note that formula \eqref{fluxViaArea} implies that the value of the cohomology class $\Flux(\Phi)$ 
on a homology class $[a]$ can be defined as the area between the curves $a$ and 
$\Phi(a)$. If $M$ is a surface of genus $\varkappa \geq 2$, then this area is well-defined. 
If $M$ is a torus, the area between $a$ and $\Phi(a)$ is defined only modulo 
the total area of $M$. However, as follows from our description of the group $L$, 
one still obtains a well-defined element of $ \Hom^1(M, \R) \,/ \,L$. In the proof of Theorem 
\ref{HamThm} below, we interchangeably use both definitions of flux, by means of formula \eqref{flux1} and as the area between curves. 
}
\end{remark}

\begin{proof}[Proof of Theorem \ref{HamThm}]
 First, let $F$ and $G$ be two simple Morse functions of $M$ such that $G = \Phi_*F$, 
 where $\Phi \in \Ham(M)$. Suppose that $\phi \colon \Gamma_F \to \Gamma_G$ is the 
 corresponding isomorphism of frozen measured Reeb graphs. 
 Then, since $\Phi$ is the lift of $\phi$, it follows from the definition of the cycles $C_{F,i}$ 
 and $C_{G,i}$ that the diffeomorphism $\Phi$ maps $C_{F,i}$ to $C_{G,i}$ 
 for each $i$. 
 Therefore, since $\Flux(\Phi) = 0$, the area between the curves $C_{F,i}$ and 
 $C_{G,i}$ is equal to zero, q.e.d.
 \par
 Conversely, assume that $F$ and $G$ are two simple Morse functions of $M$ such that 
 $G = \Phi_*F$, where $\Phi \in \SDiff_0(M)$, and that the area between the curves 
 $C_{F,i}$ and $C_{G,i}$ for all $i$ is equal to zero. 
 The area condition implies that the cohomology class $\Flux(\Phi)$ vanishes 
 on the homology classes $[C_{F,i}]$. Moreover, it follows  from the construction 
 of the curves $C_{F,i}$ that every connected component of any regular $F$-level is 
 homologous to a linear combination of classes  $[C_{F,i}]$, so  $\Flux(\Phi)$ 
 vanishes on all connected components of all regular $F$-levels. 
 Therefore, by Lemma~\ref{graphCohomology}, there exists such a $C^\infty$ 
 function $H \colon M \to \R$  that the $1$-form $H\diff F$ is closed, 
 and its cohomology class is equal to $\Flux(\Phi)$. Set
 $$
 X: = \omega^{-1}H\diff F.
 $$
 Then the flow of the vector field $X$ preserves the symplectic structure and the function $F$. Let $\Psi$ be the time-one map of $X$, and let $ \widetilde\Phi = \Phi \circ \Psi^{-1}$. 
 Then $\widetilde\Phi \in \SDiff_0(M)$, and $\widetilde\Phi_*F = G$. Furthermore, one has
\begin{align*}
 \Flux(\widetilde\Phi) &= \Flux(\Phi) - \Flux(\Psi)  = \Flux(\Phi) -  \int_{0}^1 [i_X \omega] \,\diff t =  \Flux(\Phi) - [H\diff F] = 0.
\end{align*}
 Therefore, by Theorem \ref{BanThm}, the diffeomorphism $\widetilde\Phi$ is Hamiltonian, as required.
 \end{proof}


\bigskip

\section{Related classifications results}
In this section, we discuss the relation of Theorem \ref{thm1} on the classification 
of simple Morse functions on a surface with respect to the $\SDiff$-action 
to the following two previous classification results:
\begin{longenum}
\item[(A)] Dufour, Molino, and Toulet  classified in \cite{DTM} simple Morse fibrations on surfaces with area forms under the action of symplectic diffeomorphisms.
\item[(B)] Bolsinov \cite{Bolsinov} and Kruglikov \cite{Kruglikov} classified Hamiltonian vector fields on surfaces up to the action of arbitrary diffeomorphisms.
\end{longenum}
First note that the classification provided by Theorem \ref{thm1} above is finer than either of these classifications.
Indeed, for the classification (A) we assume that two simple Morse functions $F$ and $G$ on a symplectic surface $M$ are equivalent in the sense of Theorem  \ref{thm1}, i.e. they can be obtained from each other by means of a symplectic diffeomorphism. Then the associated fibrations $\Fibr$ and $\Fibrtwo$, which are given by connected components of $F$- and $G$-levels respectively, are also symplectomorphic. Thus, equivalence of $F$ and $G$ in the sense of Theorem \ref{thm1} implies the equivalence of the associated fibrations $\Fibr$ and $\Fibrtwo$ in the sense of Dufour, Molino, and Toulet. On the other hand, the converse is not true in general: 
a symplectomorphism mapping fibrations $\Fibr$ to $\Fibrtwo$ does not have to map function 
$F$ to $G$. E.g., for  different height functions on the unit sphere, such as $z$ and $2 z$,  
their associated fibrations are the same, formed by circles of latitude on   $S^2$.
\par
Similarly, an equivalence of two simple Morse functions $F$ and $G$ in the sense of Theorem \ref{thm1} implies the equivalence of the associated Hamiltonian vector fields $\omega^{-1}\diff F$ and $\omega^{-1} \diff G$ in the sense of Bolsinov and Kruglikov. The converse statement is again not true in general, since a diffeomorphism mapping the Hamiltonian vector field $\omega^{-1}\diff F$  to the Hamiltonian vector field $\omega^{-1} \diff G$  does not have to be symplectic and does not have to map $F$ to $G$ (note that if at least one of these conditions holds, then the second condition holds as well).
\par
Thus, since the classification described in this paper is more delicate, invariants involved in the above mentioned classifications (A) and (B) should be representable in terms of the invariants of Section \ref{functionsSection}, i.e., in terms of the measured Reeb graph. For classification (A), this was already explained in Section \ref{sdiff-function-classif}. Below we briefly describe the invariants involved in classification (B) and show how they can be computed from the corresponding measured Reeb graph.
\par
Let $M$ be a surface and   $F$  a simple Morse function on \nolinebreak $M$. 
Further, let $X = \omega^{-1}\diff F$ be the  Hamiltonian vector field corresponding to $F$ 
and a certain symplectic form $\omega$ on $M$. Note that given a vector field $X$ on $M$ 
Hamiltonian with respect to some symplectic structure, its 
representation as a Hamiltonian field is not unique. 
Namely, for any  function $G(z)$ of one variable such that $G'(z) \neq 0$, we have
$ X = (\omega_G)^{-1} \diff G(F)$,
where $\omega_G:={G'(F)}\cdot \omega$ is another symplectic form on $M$. Nevertheless, the Reeb graph $\Gamma_F$ (without measure on it) is uniquely determined by the vector field $X$ and does not depend on the representation of $X$ in the form $\omega^{-1}\diff F$. Indeed, generic integral trajectories of $X$ are periodic and coincide with connected components of $F$-levels. Therefore, the fibration of $M$ into connected components of $F$-levels, and thus the Reeb graph $\Gamma_F$, can be reconstructed from $X$. In particular, if the field $X$ is diffeomorphic to another Hamiltonian vector field $Y$, then their Reeb graphs are isomorphic as abstract graphs. 
\par
To emphasize that the Reeb graph associated with a Hamiltonian vector field $X$ depends only 
on \nolinebreak $X$, but not on the particular choice of a Hamiltonian $F$, we denote this 
graph by $\Gamma_X$. Note that the graph $\Gamma_X$ is no longer endowed  with either 
the measure $\mu$, or  the function $f$, since these objects do depend on a representation 
of $X$ in the form $\omega^{-1}\diff F$ and cannot be reconstructed from the field $X$ itself. 
Instead, we have a \textit{period function} $\Pi \colon \Gamma_X \to \R \cup \{\infty \}$ 
intrinsically related to  the field $X$ and defined as follows. For each $x \in \Gamma_X$ 
that is not a vertex, the value $\Pi(x)$ is equal to the period of the trajectory of $X$ that is mapped 
to \nolinebreak $x$ under the projection $\pi \colon M \to \Gamma_X$. 
Now, if $x$ tends to a vertex $v$, then it can be shown that \nolinebreak $\Pi(x)$ has 
a finite limit for a $1$-valent vertex $v$, and it tends to infinity for a $3$-valent vertex $v$.
Thus, $\Pi$ can be extended to a continuous function $\Pi \colon \Gamma_X \to \R \cup \{\infty \}$, 
which is called the {\it period function}. Clearly, two diffeomorphic Hamiltonian vector fields have the 
same period functions. 
One should mention that  there are three kinds of invariants of Hamiltonian fields with respect to the
diffeomorphism action: invariants associated with 
edges of \nolinebreak $\Gamma_X$, invariants associated with $1$-valent vertices, and the 
so-called $\Lambda$-invariants that are associated with $3$-valent vertices. 
It turns out that all invariants of the Hamiltonian vector field $X$ 
under diffeomorphisms can be expressed in terms of the period function \nolinebreak $\Pi$, cf. 
\cite{Arnold, Bolsinov, Kruglikov}. 

\par
Now, we assume that the representation $X =\omega^{-1}\diff F$ is fixed, and show that the 
invariants of $X$ described above can be computed from the measured Reeb graph $\Gamma_F$. 
As we  mentioned above, these invariants are constructed in terms of the period function $\Pi$. 
Therefore, to show that they are expressible in terms of the measure $\mu$ and the function $f$ 
on the graph $\Gamma_F$, it suffices to express the period function \nolinebreak $\Pi$ 
in terms of $\mu$ and $f$. This can be easily done as follows: 
$$
\Pi(x) = \diffFXp{\mu([v,x])}{f(x)}\,,
$$
where $v$ is the starting point of an edge $[v,w]$ of the graph $\Gamma_F$, and $x \in [v,w]$. 
To prove that, consider formula (\ref{period-formula}) from the proof of 
Proposition \ref{measureproperty}, and note that $\Pi(x) = T(f(x))$. Thus, invariants of the 
Hamiltonian vector field $\omega^{-1}\diff F$ under diffeomorphisms can  indeed be computed 
from the measured Reeb graph $\Gamma_F$.
\par
One should note that if $F$ is any Morse function, not necessarily simple, then its Hamiltonian 
vector field $\omega^{-1}\diff F$ has additional invariants that are not expressible in terms of the 
period function, such as the so-called $\Delta$ and $Z$-invariants, 
see \cite{Bolsinov, Kruglikov}. However, if $F$ is a simple Morse function (and the set of such functions is open dense), then those 
invariants become trivial (see page 245 of \cite{BF}), and hence 
the invariants of the corresponding Hamiltonian field with respect to the symplectomorphism action 
are covered by our consideration.

\begin{remark}
Both classifications (A) and (B) have counterparts for finitely smooth fields and diffeomorphisms actions. In the $C^k$-smooth case, the corresponding invariants (such as the invariants $[\zeta_i]$ in classification (A)) become polynomials instead of infinite series, see \cite{Kruglikov}. However, in all classifications considered in this paper the invariants remain unchanged in the $C^k$-case. What is different for finite smoothness is the restriction on admissible measures $\mu$ on  Reeb graphs. In particular, both asymptotics and compatibility conditions at three-valent vertices for the series expansion of $\mu$ (see Section \ref{MRG}) now need to be satisfied only for finite number of terms depending on $k$. All the discrete invariants are the same in both cases of finite and infinite smoothness.
\end{remark}

\bibliographystyle{plain}
\bibliography{sdiff_orbits_fin}

\begin{thebibliography}{10}

\bibitem{Arnold58}
V.I. Arnold.
\newblock On the representation of functions of several variables as a
  superposition of functions of a smaller number of variables.
\newblock {\em Mat. Prosveschenie}, 3:41--61, 1958, English transl.:
  V.\,Arnold, \textit{Collected Works}, Volume 1, 25-46, 2009.

\bibitem{Arn66}
V.I. Arnold.
\newblock Sur la g{\'e}om{\'e}trie diff{\'e}rentielle des groupes de {L}ie de
  dimension infinie et ses applications {\`a} l'hydrodynamique des fluides
  parfaits.
\newblock {\em Annales de l'institut Fourier}, 16(1):319--361, 1966.

\bibitem{Arnold}
V.I. Arnold.
\newblock {\em Mathematical Methods of Classical Mechanics}.
\newblock Springer-Verlag, 1978.

\bibitem{AK}
V.I. Arnold and B.A. Khesin.
\newblock {\em Topological methods in hydrodynamics}.
\newblock Springer, New York, 1998.

\bibitem{banyaga}
A.~Banyaga.
\newblock Sur la structure du groupe des diff{\'e}omorphismes qui
  pr{\'e}servent une forme symplectique.
\newblock {\em Commentarii Mathematici Helvetici}, 53(1):174--227, 1978.

\bibitem{Bolsinov}
A.V. Bolsinov.
\newblock A smooth trajectory classification of integrable {H}amiltonian
  systems with two degrees of freedom.
\newblock {\em Sbornik: Mathematics}, 186(1):1--27, 1995.

\bibitem{BF}
A.V. Bolsinov and A.T. Fomenko.
\newblock {\em Integrable {H}amiltonian systems: geometry, topology,
  classification}.
\newblock CRC Press, 2004.

\bibitem{bolosh}
A.V. Bolsinov and A.A. Oshemkov.
\newblock Singularities of integrable {H}amiltonian systems.
\newblock In {\em Topological methods in the theory of integrable systems},
  pages 1--67. Cambridge Scientific Publishers, 2006.

\bibitem{ChSv}
A.~Choffrut and V.~{\v{S}}ver{\'a}k.
\newblock Local structure of the set of steady-state solutions to the 2{D}
  incompressible {E}uler equations.
\newblock {\em Geometric and Functional Analysis}, 22(1):136--201, 2012.

\bibitem{CDV}
Y.~Colin~de Verdi{\`e}re and J.~Vey.
\newblock Le lemme de {M}orse isochore.
\newblock {\em Topology}, 18(4):283--293, 1979.

\bibitem{DTM}
J.-P. Dufour, P.~Molino, and A.~Toulet.
\newblock Classification des syst{\`e}mes int{\'e}grables en {dimension
  \nolinebreak $2$} et invariants des mod{\`e}les de {F}omenko.
\newblock {\em Comptes rendus de l'{A}cad{\'e}mie des sciences. {S}{\'e}rie 1,
  {M}ath{\'e}matique}, 318(10):949--952, 1994.

\bibitem{earle}
C.J. Earle and J.~Eells.
\newblock A fibre bundle description of {T}eichm{\"u}ller theory.
\newblock {\em Journal of Differential Geometry}, 3(1-2):19--43, 1969.

\bibitem{MCG}
B.~Farb and D.~Margalit.
\newblock {\em A primer on mapping class groups}.
\newblock Princeton University Press, 2012.

\bibitem{gramain}
A.~Gramain.
\newblock Le type d'homotopie du groupe des diff{\'e}omorphismes d'une surface
  compacte.
\newblock {\em Annales scientifiques de l'{\'E}cole Normale Sup{\'e}rieure
  (s{\'e}rie 4)}, 6(1):53--66, 1973.

\bibitem{HT}
A.~Hatcher and W.~Thurston.
\newblock A presentation for the mapping class group of a closed orientable
  surface.
\newblock {\em Topology}, 19(3):221--237, 1980.

\bibitem{KhM}
B.~Khesin and G.~Misio{\l}ek.
\newblock Euler equations on homogeneous spaces and {V}irasoro orbits.
\newblock {\em Advances in Mathematics}, 176(1):116--144, 2003.

\bibitem{Kir1}
A.A. Kirillov.
\newblock Orbits of the group of diffeomorphisms of a circle and local {L}ie
  superalgebras.
\newblock {\em Functional Analysis and Its Applications}, 15(2):135--137, 1981.

\bibitem{kirillov1993orbit}
A.A. Kirillov.
\newblock The orbit method, {II}: Infinite-dimensional {L}ie groups and {L}ie
  algebras.
\newblock {\em Contemporary Mathematics}, 145:33--33, 1993.

\bibitem{Kruglikov}
B.S. Kruglikov.
\newblock Exact smooth classification of {H}amiltonian vector fields on
  two-dimensional manifolds.
\newblock {\em Mathematical Notes}, 61(2):146--163, 1997.

\bibitem{MSh}
V.~P. Maslov and A.~I. Shafarevich.
\newblock Asymptotic solutions of {N}avier-{S}tokes equations and topological
  invariants of vector fields and {L}iouville foliations.
\newblock {\em Theoretical and Mathematical Physics}, 180(2):967--982, 2014.

\bibitem{Moser65}
J.~Moser.
\newblock On the volume elements on a manifold.
\newblock {\em Transactions of the American Mathematical Society},
  120(2):286--294, 1965.

\bibitem{Putman}
A.~Putman.
\newblock A note on the connectivity of certain complexes associated to
  surfaces.
\newblock {\em L'Enseignement Math\'ematique}, 54:286--301, 2008.

\bibitem{Seg1}
G.~Segal.
\newblock Unitary representations of some infinite dimensional groups.
\newblock {\em Communications in Mathematical Physics}, 80(3):301--342, 1981.

\bibitem{Shn}
A.I. Shnirelman.
\newblock Lattice theory and flows of ideal incompressible fluid.
\newblock {\em Russian Journal of Mathematical Physics}, 1(1):105--113, 1993.

\bibitem{Toulet}
A.~Toulet.
\newblock {\em Classification des syst{\`e}mes int{\'e}grables en dimension 2}.
\newblock PhD thesis, Universit{\'e} de Montpellier 2, 1996.

\bibitem{Wit}
E.~Witten.
\newblock Coadjoint orbits of the {V}irasoro group.
\newblock {\em Communications in Mathematical Physics}, 114(1):1--53, 1988.

\bibitem{Wolf}
U.~Wolf.
\newblock {\em Die Aktion der Abbildungsklassengruppe auf dem Hosenkomplex}.
\newblock PhD thesis, Karlsruhe, 2009.

\end{thebibliography}

\end{document}